\documentclass[reqno,12pt]{amsart}

\usepackage[utf8]{inputenc}
\usepackage{microtype}
\usepackage{amsfonts}
\usepackage{amsmath}
\usepackage{graphicx}
\usepackage{float}
\usepackage[dvipsnames]{xcolor}
\usepackage{hyperref}
\usepackage{tikz-cd}
\usepackage{subcaption}
\usetikzlibrary{hobby}

\usepackage{amssymb}
\usepackage[inline]{enumitem}
\usepackage{mathrsfs}

\usepackage{tikz}
\usetikzlibrary{topaths}
\usetikzlibrary{calc}

\usepackage{a4wide}

\usepackage{empheq}

\newtheorem{theorem}{Theorem}[section]
\newtheorem*{theorem*}{Theorem}
\newtheorem*{maintheorem*}{Main Theorem}
\newtheorem{lemma}[theorem]{Lemma}
\newtheorem{proposition}[theorem]{Proposition}
\newtheorem*{proposition*}{Proposition}
\newtheorem{corollary}[theorem]{Corollary}
\newtheorem*{corollary*}{Corollary}

\newtheorem*{conjecture*}{Conjecture}

\newtheorem*{question*}{Question}

\newtheorem*{construction*}{Construction}

\newtheorem*{theoremA}{Theorem~\ref{thrm:concomp}}
\newtheorem*{theoremB}{Theorem~\ref{thrm:roundingcurvesinsurfaces}}
\newtheorem*{theoremC}{Theorem~\ref{thrm:classifying}}

\theoremstyle{definition}
\newtheorem{definition}[theorem]{Definition}
\newtheorem*{definition*}{Definition}
\newtheorem{remark}[theorem]{Remark}

\newcommand{\QQ}{\mathbb{Q}}
\newcommand{\RR}{\mathbb{R}}

\newcommand{\HH}{\mathbb{H}}

\DeclareMathOperator{\Con}{Con}
\DeclareMathOperator{\Iso}{Iso}

\DeclareMathOperator{\Aut}{Aut}

\DeclareMathOperator{\BAut}{bAut}
\DeclareMathOperator{\PBAut}{pbAut}

\DeclareMathOperator{\Homeo}{Homeo}

\DeclareMathOperator{\Emb}{Emb}

\DeclareMathOperator{\UConf}{UConf}
\DeclareMathOperator{\Conf}{Conf}
\DeclareMathOperator{\UEmb}{UEmb} 
 
\DeclareMathOperator{\J}{J}
\DeclareMathOperator{\RJ}{RJ}
\DeclareMathOperator{\UJ}{UJ}
\DeclareMathOperator{\RUJ}{RUJ}
\DeclareMathOperator{\CJ}{CJ}
\DeclareMathOperator{\CUJ}{CUJ}
\DeclareMathOperator{\Q}{Q}
\DeclareMathOperator{\UQ}{UQ}

\DeclareMathOperator{\interior}{int}

\numberwithin{equation}{section}

\begin{document}
	
	\title{On spaces of embeddings of circles in surfaces}
	\date{\today}
	\subjclass[2020]{Primary 58D10,55R80; 
		Secondary 20F65,20F36} 
	
	\keywords{Embedded circles, configuration space, braid group, deformation retraction}
	\author[R.~C.~Gelnett]{Ryan C. Gelnett}
	\address{Department of Mathematics and Statistics, University at Albany (SUNY), Albany, NY}
	\email{rgelnett@albany.edu}
	
	\begin{abstract}
		We consider the space of embeddings of finitely many circles that bound disks in non-positively curved surfaces. We index the connected components of this space with finite rooted trees and show that the connected components are classifying spaces of the ``braided" automorphism groups of the associated trees.
		An intermediate step to proving these results is to construct a strong deformation retract onto the subspace of geometric circles; moreover, this strong deformation retraction is equivariant with respect to transformations of the surface.
	\end{abstract}
	
	\maketitle
	\thispagestyle{empty}
	
	\section{Introduction}
	Let $X$ be a Riemannian surface of non-positive curvature.
	Our space of interest, denoted $\UJ_n(X)$, is the space of sets of $n$ unlabeled, unparameterized, disjoint simple closed curves, called Jordan configurations, in $X$ such that each curve bounds a disk in $X$; that is, the subspace of 
	$$\Emb(\sqcup_i^nS^1,X)/\Homeo(\sqcup_i^nS^1)$$
	where each component of an element bounds a disk in $X$.
	We study the homotopy type of $\UJ_n(X)$ by indexing the connected components with finite rooted trees which denote the nesting order of the curves.
	 
	{\renewcommand*{\thetheorem}{\Alph{theorem}}
	\begin{theorem}\label{thrm:concomp}
		There is a one-to-one correspondence between the collection of finite rooted trees with $n$ non-root vertices and the connected components of $\UJ_n(X)$.
	\end{theorem}}
	
	We denote by $\UJ_T(X)$ the connected component of $\UJ_n(X)$ with nesting structure given by a tree $T$. Each of these connected components can be neatly simplified to a finite dimensional subspace, which we denote $\RUJ_T(X)$, where each component of a Jordan configuration is a round circle. Write $\Iso(X)$ for the group of isometries of $X$ and $\Con(\RR^2)$ for the group of conformal transformations of $\RR^2$.
	
	{\renewcommand*{\thetheorem}{\Alph{theorem}}
	\begin{theorem}\label{thrm:roundingcurvesinsurfaces}
		There is an $\Iso(X)$-equivariant strong deformation retraction of $\UJ_n(X)$ onto $\RUJ_n(X)$, equivariant with respect to $\Con(\RR^2)$ if $X=\RR^2$.
	\end{theorem}}

	Our main result shows that $\UJ_T(X)$ is a classifying space.
	
	{\renewcommand*{\thetheorem}{\Alph{theorem}}
	\begin{theorem}\label{thrm:classifying}
		$\UJ_T(X)$ is aspherical, and its fundamental group is isomorphic to $$(\BAut(T(v_1))\times\cdots\times\BAut(T(v_m)))\rtimes B_m^{\Pi}(X).$$
	\end{theorem}}

	Here $\BAut(T)$ is the ``braided" automorphism group of the tree $T$, $B_m^{\Pi}(X)$ is a subgroup of the surface braid group $B_m(X)$, and $T(v_i)$ are the subtrees of $T$ rooted at children of the root, $v_i$; see Section~\ref{sec:homotopygroups}. 
	
	Inspiration for this paper comes from \cite{Curry2026}, where Curry, Zaremsky, and the author construct configuration spaces of circles in the plane as a generalization of configuration spaces of points. In \cite{Curry2026} we index the connected components of our spaces with isomorphism classes of finite rooted planar trees and show that these spaces are classifying spaces of the ``braided" automorphism group of these trees. In the current paper, we study simple closed curves that are not necessarily round, construct a $\Con(\RR^2)$-equivariant strong deformation retraction onto the spaces of study in \cite{Curry2026}, and consider simple closed curves in surfaces other than the plane. 
	
	Braiding a group $G$ is an idea from geometric group theory where factors of $G$ isomorphic to a subgroup of a symmetric group, $\Sigma_m$, are replaced with a subgroup of the braid group $B_m(\RR^2)$ corresponding through the natural projection $B_m(\RR^2)\to\Sigma_m$ that forgets the under/over crossing of strands and records the permutations of starting and ending points of the strands. Naive examples of braiding groups include braiding $\Sigma_m$ to get $B_m(\RR^2)$ and braiding $\Sigma_m\times\Sigma_n$ to get $B_m(\RR^2)\times B_n(\RR^2)$. More intriguing examples include braiding Thompson groups \cite{brin07,dehornoy06,brady08,funar08,funar11,bux16,witzel19,spahn}; braided Houghton groups \cite{degenhardt00,genevois22}; closely related to the groups in this paper are braided automorphism groups of an infinite, locally finite, regular tree \cite{aroca22,Cumplido01122024,skipper23,skipper24}, in fact methods in this paper may generalize to infinite, locally finite, not necessarily regular trees; and braided versions of self-similar groups and R\"over-Nekroshevych groups \cite{skipper24}.
	In this paper, we braid the automorphism groups of finite rooted trees associated to the nesting order of Jordan configurations in connected components of $\UJ_n(X)$ such that one braided factor is a subgroup of the surface braid group $B_m(X)$ and the other factor is a subgroup of the usual braid group $B_{n-m}(\RR^2)$.
	
	Simple closed curves in surfaces are useful in understanding elements of the mapping class group of surfaces; in fact the groups in Theorem~\ref{thrm:classifying} appear in \cite[Lemma~3.1]{kordek2022homomorphisms} as elements of a mapping class group that stabilize Jordan configurations nesting marked points in a surface.
	
	As a remark, in higher dimensions the story becomes more complicated due to problems including curves linking together. A common technique to better understand these spaces is to deformation retract the space of simple closed curves to round curves and compute the homotopy type of this subspace. Unlinked curves in $\RR^3$ have been studied in \cite{brendle2013configuration} and have the homotopy type of a finite-dimensional manifold. The space of a single unknotted simple closed curve in $S^3$ has been studied in \cite{hatcher1983proof} to have the homotopy type of the space of great circles in $S^3$. Progress has been made in understanding the space of linked curves in three dimensions as well; in particular, in \cite{damiani2019group} the fundamental group of the space of round Hopf links was computed to be the Quaternion group $\QQ_8$ and in \cite{boyd2025embeddinghopflink} the homotopy type of the entire space of embeddings of the Hopf link in $\RR^3$ and $S^3$ was computed to be $S^3/\QQ_8$ and $\RR P^2\times\RR P^2$, respectively.

\textbf{Organization.}
In Section 2 we define our spaces of interest and collect initial facts about them.
In Section 3 we construct trees to associate to Jordan configurations in surfaces.
In Section 4 we prove our main results involving deformation retractions between our various spaces of interest.
Finally, in Section 5 we construct a fiber bundle to compute the homotopy groups of the spaces of Jordan configurations in surfaces where the fundamental group is a ``braided" automorphism group of our trees; we also show that these spaces have the homotopy type of a CW-complex and hence, are classifying spaces.

\textbf{Acknowledgments.} Thanks are due to Justin Curry, Marco Varisco, and Matthew Zaremsky for many helpful discussions and suggestions.

	
	\section{The space of Jordan configurations}\label{sec:Jordef}
	\subsection{Definitions and Notations}
	
	In this section we construct our spaces of interest, show that configurations of circles from \cite{Curry2026} are subspaces of Jordan configurations, and recall foundational tools to be used in future sections.
	
	\begin{definition}[Configuration space of points]\label{dfn:pointconfspace}
		Let $X$ be a topological space. The \emph{labeled configuration space of $n$ points in $X$} is defined to be the subspace of $X^n$ such that each pair of entries $x_i$ and $x_j$ with $1\leq i\neq j\leq n$ are distinct, denoted $\Conf_n(*,X)$. 
		The symmetric group on $n$ elements, $\Sigma_n$, acts on $\Conf_n(*,X)$ by permuting coordinates. The orbit space $\UConf_n(*,X):=\Conf_n(*,X)/\Sigma_n$ is called the \emph{(unlabeled) configuration space of $n$ points in $X$}.
	\end{definition}
	
	\begin{definition}[Configuration space of circles]\label{dfn:circleconfspace}
		Let $\Conf_n(S^1,\RR^2)$ be the subspace of $\left(\RR^3\right)^n$ such that for each pair of entries $(x_i,y_i,r_i)$ and $(x_j,y_j,r_j)$ with $1\leq i\neq j\leq n$ we have $r_i,r_j>0$ and either:
		\begin{center}
			$(r_i-r_j)^2-((x_i-x_j)^2+(y_i-y_j)^2)>0$ \hfil or\hfil $(r_i+r_j)^2-((x_i-x_j)^2+(y_i-y_j)^2)<0$.
		\end{center}
		We call this space the \emph{labeled configuration space of $n$ circles in the plane}. The symmetric group on $n$ elements, $\Sigma_n$, acts on $\Conf_n(S^1,\RR^2)$ by permuting coordinates. The orbit space $\UConf_n(S^1,\RR^2):=\Conf_n(S^1,\RR^2)/\Sigma_n$ is called the \emph{(unlabeled) configuration space of $n$ circles in the plane}, where the elements of $\UConf_n(S^1,\RR^2)$ are sets $\{(x_1,y_1,r_1),\dots,(x_n,y_n,r_n)\}$ where each pair of elements of the set satisfies the above inequalities.
	\end{definition}
	
	In this definition, the $x_i$ and $y_i$ coordinates give the center of a circle and $r_i$ gives the radius of a circle. The inequalities are derived from the known conditions for a pair of circles to intersect,
	$$(r_i-r_j)^2\leq(x_i-x_j)^2+(y_i-y_j)^2\leq(r_i+r_j)^2,$$
	so $\Conf_n(S^1,\RR^2)$ and $\UConf_n(S^1,\RR^2)$ are the collections of $n$-tuples and $n$-element sets of pairwise disjoint circles.
	
	The primary difference between the configuration spaces of points and configuration spaces of circles is that circles must be \emph{disjoint} rather than merely \emph{distinct}. 
	The impact this has on the configuration spaces of circles is that for a given circle, $C_i=(x_i,y_i,r_i)$, another circle, $C_j=(x_j,y_j,r_j)$, may lie in the unbounded region surrounding $C_i$ or the region bounded by $C_i$. 
	When $C_j$ lies in the region bounded by $C_i$, we say $C_j$ is \emph{nested} in $C_i$ and $C_i$ \emph{nests} $C_j$. If either $C_i$ or $C_j$ nests the other, we have $(r_i-r_j)^2-((x_i-x_j)^2+(y_i-y_j)^2)>0$. If $C_i$ and $C_j$ each lie in the unbounded regions surrounding each other, then $(r_i+r_j)^2-((x_i-x_j)^2+(y_i-y_j)^2)<0$.
	
	For the rest of the paper, let $X$ be a connected Riemannian surface of non-positive curvature. This choice is due to the fact that we will only consider curves that bound a disk and we want this disk to be unique; this does not happen when the surface has positive curvature. We will assume that any choice $X$ comes with a chosen metric, sometimes denoted $d_X$. Recall that an \emph{embedding} is a function that is a homeomorphism onto its image.
	
	\begin{definition}[Embeddings of circles]\label{dfn:embofcircles}
		Define $\Emb_n(S^1,X)$ to be the topological space whose elements are ordered $n$-tuples $(f_1,\dots,f_n)$ of injective maps $f_i:S^1\to X$ with pairwise disjoint images. Since $S^1$ is compact and $X$ is Hausdorff, elements of $\Emb_n(S^1,X)$ are $n$-tuples of embeddings. The topology on $\Emb_n(S^1,X)$ comes from realizing this space as a subspace of $\left(\Emb(S^1,X)\right)^n$, where $\Emb(S^1,X)$ is the space of all embeddings of a circle into $X$ with the compact-open topology. We call $\Emb_n(S^1,X)$ the \textit{labeled embedding space of $n$ circles in $X$}. There is an action of the symmetric group $\Sigma_n$ on $\Emb_n(S^1,X)$ given by permuting the entries in the tuples, we call the orbit space the \textit{(unlabeled) embedding space of $n$ circles in $X$}, and denote it $\UEmb_n(S^1,X)$. We can view the elements of $\UEmb_n(S^1,X)$ as (unordered) sets $\{f_1,\dots,f_n\}$ of $n$ embeddings of the circle in $X$ with pairwise disjoint images.
	\end{definition}
	
	An alternative definition of $\Emb_n(S^1,X)$ is $\Emb\left(\sqcup_i^n S^1,X \right)$ with the compact-open topology. 
	
	We are more interested in the geometry of these curves, so we pass to unparameterized curves by quotienting out by the action of precomposing with homeomorphisms of the circle. 
	The orbit space $\Q_n(X):=\Emb_n(S^1,X)/(\Homeo(S^1))^n$ is the space of $n$-tuples of equivalence classes $([f_1],\dots,[f_n])$ with $(f_1,\dots,f_n)$ an element of $\Emb_n(S^1,X)$. Since the action of $(\Homeo(S^1))^n$ on $\Emb_n(S^1,X)$ does not permute the entries, we can pass to the orbit space $\UQ_n(X):=\Q_n(X)/\Sigma_n$ where elements are sets of $n$ equivalence classes $\{[f_1],\dots,[f_n]\}$ where $([f_1],\dots,[f_n])$ is an element of $\Q_n(X)$.
	Furthermore, in this paper we will only consider curves that bound disks in the following sense.
	
	\begin{definition}
		Let $B^2$ be the closed unit disk in the plane centered at the origin $o=(0,0)$. We say $D\subset X$ is a \emph{Jordan domain} if it forms the image of a continuous injective map $B^2\to X$, called a \emph{parametrization} of $D$. Note that since $B^2$ is compact, these continuous injective maps are embeddings. We write $\Emb(B^2,X)$ for the space of embeddings of $B^2$ in $X$ with the compact-open topology. The \emph{space of Jordan domains}, denoted $\mathcal{D}(X)$, is the quotient space $\Emb(B^2,X)/\sim$ where two maps are identified if they have same image, that is, $f\sim g$ if and only if $f(B^2)=g(B^2)$.
	\end{definition}
	
	\begin{definition}[Jordan configurations]\label{dfn:jrdncurves}
		Consider the subspace of $\Q_n(X)$ such that each entry of an $n$-tuple is the boundary of an element of $\mathcal{D}(X)$, we call this the \textit{labeled configuration space of $n$ Jordan curves in $X$}, we call each element a \emph{labeled Jordan configuration}, and denote the space $\J_n(X)$. We call the image of $\J_n(X)$ in $\UQ_n(X)$ the \textit{(unlabeled) configuration space of $n$ Jordan curves in $X$}, we denote it $\UJ_n(X)$, and we call each element a \emph{Jordan configuration}.
	\end{definition}
	We get a commutative diagram, Figure~\ref{fig:comdiaemb}. In the case of the plane, the Jordan curve theorem assures $\UQ_n(\RR^2)=\UJ_n(\RR^2)$.
	\begin{figure}[h]
	\centering
	\begin{tikzcd}
		\Emb_n(S^1,X)\arrow{r}\arrow{d}&\Q_n(X)\arrow{d}&\arrow[hook']{l}\J_n(X) \arrow{d} \arrow[hook]{r} & (\mathcal{D}(X))^n \arrow{d} & \arrow{l} (\Emb(B^2,X))^n\arrow{d}\\
		\UEmb_n(S^1,X)\arrow{r}&\UQ_n(X)&\arrow[hook']{l}\UJ_n(X) \arrow[hook]{r} & (\mathcal{D}(X))^n/\Sigma_n & \arrow{l}(\Emb(B^2,X))^n/\Sigma_n\
	\end{tikzcd}
	\caption{}\label{fig:comdiaemb}
	\end{figure}

	\subsection{Initial Facts}
	
	Recall a function, $\varphi:X\to Y$, between topological spaces $X$ and $Y$ is \emph{sequentially continuous} if for every convergent sequence $(x_n)\to x$ in $X$ we have $(\varphi(x_n))$ converges to $\varphi(x)$. Further, a topological space $X$ is \emph{sequential} if for every topological space $Y$, a function $\varphi:X\to Y$ is continuous if and only if $\varphi$ is sequentially continuous. For example, every metric space is sequential.
	
	\begin{lemma}\label{lem:sequential}
		$\J_n(X)$ and $\UJ_n(X)$ are sequential.
	\end{lemma}
	\begin{proof}
		$\Emb(S^1,X)$ and $\Emb(B^2,X)$ are endowed with the compact-open topology and since $S^1$ and $B^2$ are compact and $X$ is a Riemann surface, Proposition~\cite[A.13]{hatcher02} shows that the metric topology defined by the metric $d_{\Emb}(f,g):=\sup_{x\in S^1} d_{X}(f(x),g(x))$, where $d_{X}$ is any fixed metric on $X$, 
		coincides with the compact-open topology on $\Emb(S^1,X)$ and $\Emb(B^2,X)$. Since $\Emb_n(S^1,X)$ is a subspace of $\left(\Emb(S^1,X)\right)^n$, it too is metrizable. $\Q_n(X)$, $(\mathcal{D}(X))^n$, $\UQ_n(X)$, and $(\mathcal{D}(X))^n/\Sigma_n$ are quotients of $\Emb_n(S^1,X)$ and $(\Emb(B^2,X))^n$ and hence are sequential by Corollary~\cite[1.14]{Franklin1965}. Therefore $\J_n(X)$ and $\UJ_n(X)$ are sequential.
	\end{proof}
	
	\begin{definition}[Round Jordan configurations]\label{dfn:rndjrdncrves}
		We define a curve in $\J_1(X)$ to be \emph{round} if its image is equidistant to a point in $X$. Let $\RJ_n(X)$ be the subspace of $\J_n(X)$ where each coordinate is round, we call this space the \textit{labeled configuration space of round Jordan curves in $X$}. Define $\RUJ_n(X)$ to be the image of $\RJ_n(X)$ as the subspace of $\UJ_n(X)$ called the \textit{(unlabeled) configuration space of round Jordan curves in $X$}. 
	\end{definition}
	
	In fact, for $X=\RR^2$ with the standard Euclidean metric, these are the labeled and unlabeled configuration spaces of circles in the plane from~\cite{Curry2026}.
	
	\begin{theorem}\label{thm:embeddingconfspaces}
		$\Conf_n(S^1,\RR^2)$ and $\UConf_n(S^1,\RR^2)$ embed into $\J_n(\RR^2)$ and $\UJ_n(\RR^2)$, respectively, as $\RJ_n(\RR^2)$ and $\RUJ_n(\RR^2)$.
	\end{theorem}
	\begin{proof}
		Define $\Psi:\Conf_n(S^1,\RR^2)\to \J_n(\RR^2)$ to be the map taking $((x_1,y_1,r_1),\dots,(x_n,y_n,r_n))$ to
		$$\left([(r_1\cos(t)+x_1,r_1\sin(t)+y_1)],\dots,[(r_n\cos(t)+x_n,r_n\sin(t)+y_n)]\right).$$ 
		The map $\Psi$ is $\Sigma_n$-equivariant and so induces a map $\overline{\Psi}:\UConf_n(S^1,\RR^2)\to \J_n(\RR^2)$. 
		Both maps are injective since circles and round (labeled) Jordan configurations are uniquely determined by their centers and radii. It is sufficient to show $\Psi$ is an embedding for $n=1$.
		
		As seen in the proof of Lemma \ref{lem:sequential}, $\Emb(S^1,\RR^2)$ is metrizable with the metric $d_{\Emb}(f,g):=\sup_{t\in S^1} d_{\RR^2}(f(t),g(t))$ induced by the Euclidean metric, $d_{\RR^2}$. For $d_{\Conf_n(S^1,\RR^2)}$ we choose the sum of the two dimensional Euclidean metric for the centers and the one dimensional Euclidean metric for the radii, that is 
		$$d_{\Conf_n(S^1,\RR^2)}((a,b,c),(a',b',c')):=d_{\RR^2}((a,b),(a',b'))+|c-c'|.$$
		Consider the map $\hat{\Psi}:\Conf_1(S^1,\RR^2)\to\Emb(S^1,\RR^2)$ defined by $(x,y,r)\mapsto(r\cos(t)+x,r\sin(t)+y)$, the following calculation of $d_{\Emb}(\hat{\Psi}(x,y,r),\hat{\Psi}(x',y',r'))$ shows $\hat{\Psi}$ preserves distances and hence is an embedding:
		\begin{align*}
			&\sup_{t\in S^1} \left[d_{\RR^2}((r'\cos(t)+x',r'\sin(t)+y'),(r\cos(t)+x,r\sin(t)+y))\right]\\
			&=\sup_{t\in S^1}\sqrt{((r'\cos(t)+x')-(r\cos(t)+x))^2+((r'\sin(t)+y')-(r\sin(t)+y))^2}\\
			&=\sup_{t\in S^1}\sqrt{((r'-r)\cos(t)+(x'-x))^2+((r'-r)\sin(t)+(y'-y))^2}\\
			&=\sup_{t\in S^1}\sqrt{(r'-r)^2+2(r'-r)(\cos(t)(x'-x)+\sin(t)(y'-y))+((x'-x)^2+(y'-y)^2)}\\
			&=\sqrt{(r'-r)^2+2|r'-r|\sqrt{(x'-x)^2+(y'-y)^2}+(x'-x)^2+(y'-y)^2}\\
			&=|r'-r|+\sqrt{(x'-x)^2+(y'-y)^2}.
		\end{align*}
		The quotient map $\Emb(S^1,\RR^2)\to \J_1(\RR^2)$ restricted to the image of the isometry $\hat{\Psi}$ is injective and hence an embedding. Post-composing the isometry $\hat{\Psi}$ with this embedding is the embedding $\Psi$.
	\end{proof}
	
	One last fact we recall is a strengthening of the uniformization theorem.
	
	\begin{theorem}[Killing-Hopf Theorem~\cite{hopf2001clifford,killing1891ueber}]\label{killhopfthrm}
		Let $X$ be a complete connected Riemannian surface with constant curvature. Then $X$ is isometric to a quotient of the Euclidean plane, hyperbolic plane, or Riemann sphere by a discontinuous, fixed point free group of isometries. We write $\widetilde{X}$ for this universal cover.
	\end{theorem}

	
	\subsection{Equivariant centers and parametrizations of their domains}
	
	Let $f=\{[f_1],\dots,[f_n]\}$ be a Jordan configuration in $X$.
	Since $X$ is not a sphere, we can associate to any curve, $f_i$, two possibly equal closed subsets of $X$: a simply connected domain bounded by $f_i$, call this $D(f_i)$; and the closure of the connected component of $X\setminus f$ whose interior non-trivially intersects $D(f_i)$ and has $f_i$ as a boundary component, call it $R(f_i)$.  
	The interior of a set $D$ is denoted $\interior(D)$. We call $\interior(D(f_i))$ the \emph{interior of the curve} $f_i$, see Figure~\ref{fig:multicurv_regions}. 
	
	\begin{figure}[htb]
		\centering
		\begin{tikzpicture}[line width=1pt]
			
			\draw (-2.2,-1) circle (0.3cm);
			
			\draw[use Hobby shortcut,closed=true, scale=1.3, rotate=55]
			(-1.8,-.2) .. (-1.6,.5) .. (-1,.9) .. 
			(-.3,1) .. (1,.7) .. (2,.2) .. (2.22,0) ..
			(2.42,-.45) .. (2.25,-1) ..
			(1.8,-1.18) .. (1.2,-1) ..
			(0,-.8) ..
			(-.9,-.9) .. (-1.5,-.7);
			\draw[use Hobby shortcut,closed=true, scale=1.3, rotate=45,xshift=-10]
			(.6,0) .. (0,.2) .. (-1,-.5) .. (-.5,-.5) .. (-.5,0) .. (1,0);
			\draw[use Hobby shortcut,closed=true, scale=1, rotate=180, yshift=-50, xshift=-55]
			(1.8,0.9) .. (1.8,1.3) .. (2,1.3) .. (2,1) .. (1,0) .. (.8,.2) .. (1.5,1);
			
			\node at (-2.2,-1) {$1$};
			\node at (-0.75,-1.3) {$2$};
			\node at (2,.3) {$3$};
			\node at (.8,1.2) {$4$};

			\begin{scope}[xshift=6cm]
				
				\fill[use Hobby shortcut,closed=true,gray, scale=1.3, rotate=55]
				(-1.8,-.2) .. (-1.6,.5) .. (-1,.9) .. 
				(-.3,1) .. (1,.7) .. (2,.2) .. (2.22,0) ..
				(2.42,-.45) .. (2.25,-1) ..
				(1.8,-1.18) .. (1.2,-1) ..
				(0,-.8) ..
				(-.9,-.9) .. (-1.5,-.7);
				\draw[use Hobby shortcut,closed=true, scale=1.3, rotate=55]
				(-1.8,-.2) .. (-1.6,.5) .. (-1,.9) .. 
				(-.3,1) .. (1,.7) .. (2,.2) .. (2.22,0) ..
				(2.42,-.45) .. (2.25,-1) ..
				(1.8,-1.18) .. (1.2,-1) ..
				(0,-.8) ..
				(-.9,-.9) .. (-1.5,-.7);
				
				\node at (2,.3) {$3$};
				
			\end{scope}
			
			\begin{scope}[xshift=12cm]
				\fill[use Hobby shortcut,closed=true,gray, scale=1.3, rotate=55, even odd rule]
				(-1.8,-.2) .. (-1.6,.5) .. (-1,.9) .. 
				(-.3,1) .. (1,.7) .. (2,.2) .. (2.22,0) ..
				(2.42,-.45) .. (2.25,-1) ..
				(1.8,-1.18) .. (1.2,-1) ..
				(0,-.8) ..
				(-.9,-.9) .. (-1.5,-.7);
				\fill[white,use Hobby shortcut,closed=true, scale=1.3, rotate=45,xshift=-10]
				(.6,0) .. (0,.2) .. (-1,-.5) .. (-.5,-.5) .. (-.5,0) .. (1,0);
				\fill[white,use Hobby shortcut,closed=true, scale=1, rotate=180, yshift=-50, xshift=-55]
				(1.8,0.9) .. (1.8,1.3) .. (2,1.3) .. (2,1) .. (1,0) .. (.8,.2) .. (1.5,1);
				
				\draw[use Hobby shortcut,closed=true, scale=1.3, rotate=55]
				(-1.8,-.2) .. (-1.6,.5) .. (-1,.9) .. 
				(-.3,1) .. (1,.7) .. (2,.2) .. (2.22,0) ..
				(2.42,-.45) .. (2.25,-1) ..
				(1.8,-1.18) .. (1.2,-1) ..
				(0,-.8) ..
				(-.9,-.9) .. (-1.5,-.7);
				\draw[use Hobby shortcut,closed=true, scale=1.3, rotate=45,xshift=-10]
				(.6,0) .. (0,.2) .. (-1,-.5) .. (-.5,-.5) .. (-.5,0) .. (1,0);
				\draw[use Hobby shortcut,closed=true, scale=1, rotate=180, yshift=-50, xshift=-55]
				(1.8,0.9) .. (1.8,1.3) .. (2,1.3) .. (2,1) .. (1,0) .. (.8,.2) .. (1.5,1);
				
				\node at (-0.75,-1.3) {$2$};
				\node at (2,.3) {$3$};
				\node at (.8,1.2) {$4$};
				
			\end{scope}
			
		\end{tikzpicture}
		\caption{From left to right: a labeled Jordan configuration $f$, $D(f_3)$, and $R(f_3)$.}
		\label{fig:multicurv_regions}
	\end{figure}
	
	\begin{lemma}[Carath\'eodory's Theorem{\cite{caratheodory1913gegenseitige}} {\cite[Thm.~5.1.1]{krantz2006geometric}} {\cite[Thm.~2.11]{pommerenke2013boundary}}]\label{lem:conformalmapping}
		Let $\widetilde{X}$ be $\RR^2$, $\HH^2$, or $S^2$. 
		For any domain $D\in \mathcal{D}(\widetilde{X})$, point $p\in \interior(D)$, and a non-trivial vector $u$ in the tangent plane of $\widetilde{X}$ at $p$, denoted $T_p\widetilde{X}$, there exists a unique continuous map $\gamma_D = \gamma_{D,p,u}\in \Emb(B^2,\widetilde{X})$ with $\gamma_D(B^2) = D$, such that $\gamma_D$ is conformal on $\interior(B^2)$, $\gamma_D(o) = p$, and the derivative of $\gamma_D$ evaluated at $o$ in the direction of $(1,0)$, denoted $d(\gamma_D)_o((1, 0))$, is parallel to $u$. Furthermore, if $D_i\in \mathcal{D}(\widetilde{X})$ is a sequence of domains containing $p$ and converging to $D$, then $\gamma_{D_i,p,u}$ converges to $\gamma_D$.
	\end{lemma}
	
	\begin{figure}[htb]
		\centering
		\begin{tikzpicture}[line width=1pt]
			
			\draw[line width=2,use Hobby shortcut,closed=true, scale=1.3]
			(-1.8,-.2) .. (-1.6,.5) .. (-1,.9) .. 
			(-.3,1) .. (1,.7) .. (2,.2) .. (2.22,0) ..
			(2.42,-.45) .. (2.25,-1) ..
			(1.8,-1.18) .. (1.2,-1) ..
			(0,-.8) ..
			(-.9,-.9) .. (-1.5,-.7);
			\draw[use Hobby shortcut,closed=true, scale=.8,xshift=6.5, yshift=3.5]
			(-1.8,-.2) .. (-1.6,.5) .. (-1,.9) .. 
			(-.3,1) .. (1,.7) .. (2,.2) .. (2.22,0) ..
			(2.42,-.45) .. (2.25,-1) ..
			(1.8,-1.18) .. (1.2,-1) ..
			(0,-.8) ..
			(-.9,-.9) .. (-1.5,-.7);
			\draw[use Hobby shortcut,closed=true, scale=.4,xshift=30, yshift=12]
			(-1.8,-.2) .. (-1.6,.5) .. (-1,.9) .. 
			(-.3,1) .. (1,.7) .. (2,.2) .. (2.22,0) ..
			(2.42,-.45) .. (2.25,-1) ..
			(1.8,-1.18) .. (1.2,-1) ..
			(0,-.8) ..
			(-.9,-.9) .. (-1.5,-.7);
			
			\draw[use Hobby shortcut,closed=false, scale=1.3]
			(2.4,-.7) .. (.5,0.1) .. (-1.8,-0.1);
			
			\draw[use Hobby shortcut,closed=false, scale=1.3]
			(1.8,-1.17) .. (.5,0.1) .. (-1.4,.7);
			
			\draw[use Hobby shortcut,closed=false, scale=1.3]
			(1.6,.45) .. (.5,0.1) .. (-1.15,-.85);
			
			\draw[use Hobby shortcut,closed=false, scale=1.3]
			(.3,-.8) .. (.5,0.1) .. (.3,.92);
			
			\node at (-1.5,1.5) {$[f]$};
			\node at (0.45,0.35) {$p$};
			\filldraw (.645,0.13) circle (2.5pt);
			
		\end{tikzpicture}
		\caption{Conformal parametrization of the interior of a curve $[f]$ with center $p$ given by Carath\'eodory's Theorem.}
		\label{fig:parameterized_domain}
	\end{figure}
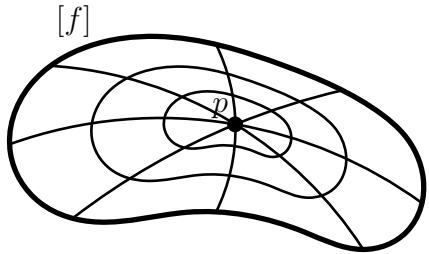
	
	Any sequence of domains containing a point $p$ lift to a unique sequence of domains in $\widetilde{X}$ containing a chosen point $\widetilde{p}$ in the fiber of $p$. Each pair of points in the fiber of $p$ are related by an isometry of $\widetilde{X}$ that sends their corresponding sequence of domains to each other and preserves their image in $X$, the original sequence of domains. An isometry of $\widetilde{X}$ that preserves the images of the domains will send the parametrization of each domain in a sequence to the parametrization of a domain in another sequence, modulo the rotation of the parametrizations. Hence, the above result holds for $\J_1(X)$ when $p\in \interior(D(f_i))$ for $[f_i]$ a curve.
	The last statement of the lemma implies the map, $\J_1(X)\to\Emb(B^2,X)$ defined by $[f_i]\mapsto \gamma_{D(f_i),p,u}$ is continuous since $\J_1(X)$ is sequential by Lemma~\ref{lem:sequential}. 
	
	\begin{definition}[A center of a Jordan domain]\label{dfn:center}
		Define a \textit{center on a set $A\subseteq \mathcal{D}(X)$} to be a continuous map $c$ from $A$ to $X$ such that $c(D)\in \interior(D)$ for all $D\in A$.
	\end{definition}
	
	Let $\Con(X)$ be the group of conformal transformations of $X$ and $\Iso(X)$ be the group of isometries of $X$. When $X=\HH^2$ it is well known that $\Con(X)=\Iso(X)$. 
	
	\begin{theorem}[{\cite[Thm.~1.1]{belegradek2025point}}]\label{thrm:extequcenter}
		For $M$ a connected Riemannian surface and $G\leq\Iso(M)$, any $G$-equivariant center on a closed $G$-invariant subspace of domains in $M$ can be extended to a $G$-equivariant center on $\mathcal{D}(M)$.
	\end{theorem}
	
	In fact, if $M=\RR^2$, then $G$ can include more transformations than isometries.
	
	\begin{theorem}[{\cite[Thm.~3.10]{belegradek2025point}}]\label{thrm:extequcenterplane}
		Let $A\subset \mathcal{D}(\RR^2)$ be a closed $\Con(\RR^2)$-invariant set.
		Then any $\Con(\RR^2)$-equivariant center on $A$ may be extended to a $\Con(\RR^2)$-equivariant center on $\mathcal{D}(\RR^2)$.
	\end{theorem}
	
	\begin{definition}[Round Jordan domains]\label{dfn:rndjrdndomain}
		Define a Jordan domain $D\in \mathcal{D}(X)$ to be \textit{round} if there is a point $c_D\in D$ which has constant distance from all points of $\partial D$. Let $R\mathcal{D}(X)\subset \mathcal{D}(X)$ denote the \textit{space of round Jordan domains in $X$}. 
	\end{definition}
	
	\begin{lemma}\label{lem:existingcenter}
		The map $c:R\mathcal{D}(X)\to X$ defined by $c(D)=c_D$ is an $\Iso(X)$-equivariant center on the closed $\Iso(X)$-invariant set $R\mathcal{D}(X)$. When $X=\RR^2$, this is a $\Con(X)$-equivariant center on the closed $\Con(X)$-invariant set $R\mathcal{D}(X)$.
	\end{lemma}
	\begin{proof}
		Since the limit of a sequence of circles is either a circle or a point and a point is not a Jordan domain, $R\mathcal{D}(X)$ is a closed set. When the limit of a sequence of circles is a circle, the corresponding sequence of centers converges to the center of the limit circle, so $c$ is a continuous map. Isometries send centers of circles to centers of circles.
		
		The group $\Con(\RR^2)$ is generated by affine maps, i.e. rotations, uniform scaling \cite{b09efa93-f6bd-3c98-9d77-3cadda8bfa53}, and translations which preserve circles and centers of circles, so the lemma is true.
	\end{proof}
	
	\begin{corollary}\label{cor:existcenter}
		There exist centers on $X$ that are $\Iso(X)$-equivariant and coincide with the canonical center of circles; these centers are $\Con(X)$-equivariant if $X=\RR^2$.
	\end{corollary}
	\begin{proof}
		This is immediate from Theorem~\ref{thrm:extequcenterplane} and Lemma~\ref{lem:existingcenter}.
	\end{proof}
	
	\section{Trees and path connected components}
	
		When considering paths between Jordan configurations, it is immediately evident that no curve may become nested or un-nested; this is due to the Jordan curve theorem and the fact that a path must be a Jordan configuration at every moment of the path: a curve $[f_i]$ splits the plane into two regions; if curve $[f_j]$ lies in one of these regions, it must pass through $[f_i]$ to get to the other region, which is impossible since curves in a Jordan configuration are pairwise disjoint. This phenomenon tells us for $n\geq2$ we have a disconnected space with one connected component homotopy equivalent to a configuration space of points. The case of configuration spaces of circles in the plane was captured in \cite{Curry2026}.
	
	\begin{proposition}[\cite{Curry2026}]\label{prop:hetoconfofpoints}
		The connected components of $\Conf_n(S^1,\RR^2)$ and $\UConf_n(S^1,\RR^2)$ where no circles are nested are homotopy equivalent to $\Conf_n(*,\RR^2)$ and $\UConf_n(*,\RR^2)$, respectively.
	\end{proposition}
	
	To visualize the collection of connected components of this space we associate each component with a tree.
	
	\begin{definition}[Tree]\label{dfn:trees}
		In this paper, a \emph{tree} is a connected graph that is \begin{enumerate*}
			\item finite,
			\item directed,
			\item rooted with the root having zero incoming edges and any finite number of outgoing edges,
			\item non-root vertices have a single incoming edge and any finite number of outgoing edges,
			\item no cycles.
		\end{enumerate*}
		A \emph{labeled tree} is a tree along with a bijective function, $f$, from $\{1,\dots,n\}$ to the non-root vertices, often we say vertex $i$ instead of $f(i)$.
	\end{definition}
		
		We call the terminal vertex of an edge a \emph{child} of the initial vertex of the edge, and the collection of all terminal vertices of all edges with the same initial point the \emph{children} of the initial vertex. When we fix a left to right planar ordering of the children of each vertex we get a \emph{planar (labeled) tree} and we can define an isomorphism and understand equality. An \emph{isomorphism of planar trees} is a bijection of vertex and edge sets where the direction of the edges are preserved, and labels of vertices are preserved for an isomorphism of labeled trees. Note an isomorphism of (labeled) trees need not preserve the ordering of the children of a vertex; in fact, an isomorphism that preserves the planar ordering is unique, this is our sense of equality.
	
	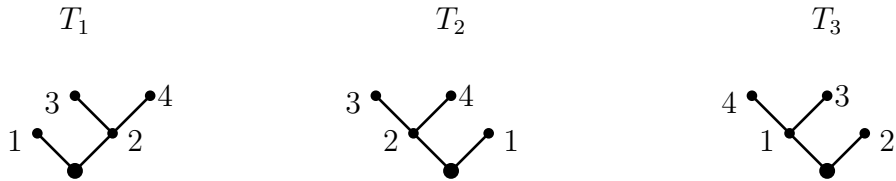
\begin{figure}[htb]
		\centering
		\begin{tikzpicture}[line width=1pt]
			\draw (0,0) -- (0.5,-0.5) -- (1.5,0.5)   (1,0) -- (.5,.5);
			
			\filldraw (0,0) circle (1.5pt);
			\filldraw (0.5,-0.5) circle (2.5pt);
			\filldraw (1.5,0.5) circle (1.5pt);
			\filldraw (1,0) circle (1.5pt);
			\filldraw (0.5,0.5) circle (1.5pt);
			
			\node at (-0.3,-0.1) {$1$};
			\node at (1.3,-0.1) {$2$};
			\node at (0.2,0.4) {$3$};
			\node at (1.7,0.5) {$4$};
			
			\node at (0.5,1.5) {$T_1$};
			
			\begin{scope}[xshift=6cm]
				\draw (0,0) -- (-0.5,-0.5) -- (-1.5,.5)	(-1,0) -- (-0.5,0.5);
				
				\filldraw (0,0) circle (1.5pt);
				\filldraw (-0.5,-0.5) circle (2.5pt);
				\filldraw (-1,0) circle (1.5pt);
				\filldraw (-0.5,0.5) circle (1.5pt);
				\filldraw (-1.5,0.5) circle (1.5pt);
				
				\node at (0.3,-0.1) {$1$};
				\node at (-1.3,-0.1) {$2$};
				\node at (-1.8,0.4) {$3$};
				\node at (-0.3,0.5) {$4$};
				
				\node at (-0.5,1.5) {$T_2$};
			\end{scope}
			
			\begin{scope}[xshift=11cm]
				\draw (0,0) -- (-0.5,-0.5) -- (-1.5,.5) 			(-1,0) -- (-0.5,0.5);
				
				\filldraw (0,0) circle (1.5pt);
				\filldraw (-0.5,-0.5) circle (2.5pt);
				\filldraw (-1,0) circle (1.5pt);
				\filldraw (-0.5,0.5) circle (1.5pt);
				\filldraw (-1.5,0.5) circle (1.5pt);
				
				\node at (0.3,-0.1) {$2$};
				\node at (-1.3,-0.1) {$1$};
				\node at (-1.8,0.4) {$4$};
				\node at (-0.3,0.5) {$3$};
				
				\node at (-0.5,1.5) {$T_3$};
			\end{scope}
			
		\end{tikzpicture}
		\caption{Three labeled trees, $T_1$, $T_2$, and $T_3$, where edges are directed upwards. As labeled trees, $T_1$ and $T_2$ are equal, and $T_3$ is not equal to either of them. The underlying unlabeled trees of $T_1$, $T_2$, and $T_3$ are equal. As planar labeled trees, $T_1$ and $T_2$ are isomorphic but not equal, and $T_3$ is not isomorphic to either of them. The underlying planar unlabeled tree of $T_3$ is equal to that of $T_2$, and is isomorphic but not equal to that of $T_1$.}
		\label{fig:iso_nonequal_labeled_config_tree}
	\end{figure}
	
	Next we recall the correspondence between planar (labeled) trees and (labeled) configurations of circles in the plane that was set up in \cite{Curry2026}. Here planar trees are used to capture the lexicographical ordering of the centers of the circles.
	Let $\kappa=(C_1,\dots,C_n)$ be an element of $\Conf_n(S^1,\RR^2)$. Define the \emph{planar labeled tree} of $\kappa$, denoted $T_{\kappa}$, to be such that a vertex $j$ is a child of a vertex $i$ if and only if $C_j$ is nested by $C_i$ and no other $C_k$ simultaneously nests $C_j$ and is nested by $C_i$, and a vertex $i$ is a child of the root, $\varnothing$, if and only if $C_i$ is not nested inside any other $C_k$. The unlabeled situation is handled by giving a configuration of circles an arbitrary labeling, constructing a planar labeled tree as before, and finally forgetting the arbitrary labeling.
	In the other direction, to construct a fixed (labeled) configuration corresponding to a (labeled) tree we need to be more precise with the placement of circles such that the planar ordering of the vertices of the tree corresponds to the lexicographical ordering of the centers of circles. From this correspondence, constructed in detail in \cite{Curry2026}, we achieve our goal of indexing the connected components of the labeled and unlabeled configuration spaces of circles in the plane.
	
	\begin{proposition}[Curry, Gelnett, Zaremsky, \cite{Curry2026}]\label{prop:components}
		Let $T$ and $U$ be planar trees with $n$ non-root vertices. Then the configurations corresponding to $T$ and $U$ lie in the same connected component of $\UConf_n(S^1,\RR^2)$ if and only if $T\cong U$. Moreover, there exists a planar tree for every connected component of $\Conf_n(S^1,\RR^2)$ such that the corresponding configuration is contained in the connected component, and the planar tree $T_{\kappa}$ corresponding to any configuration $\kappa$ in the connected component indexed by $T$ is isomorphic to $T$. The analogous result holds for the labeled case.
	\end{proposition}
	
	More simply put, the constructions induce a one-to-one correspondence between the isomorphism classes of planar trees and connected components of the configuration space of circles in the plane.
	Hence, the following is well defined:
	
	\begin{definition}[Space of $T$-configurations]\label{def:Tconfigs}
		For $T$ a planar labeled tree with $n$ non-root vertices, the \emph{space of labeled $T$-configurations} of $n$ circles in the plane, denoted $\Conf_T(S^1,\RR^2)$, is the connected component of $\Conf_n(S^1,\RR^2)$ containing the configurations corresponding to $T$. Similarly, for $T$ a planar (unlabeled) tree with $n$ non-root vertices, the \emph{space of $T$-configurations} of $n$ circles in the plane, denoted $\UConf_T(S^1,\RR^2)$, is the connected component of $\UConf_n(S^1,\RR^2)$ containing the configurations corresponding to $T$.
	\end{definition}
	
	Let $\kappa=\{C_1,\dots,C_n\}$ be a configuration of circles in the plane and identify $\kappa$ with its image in $\UJ_n(\RR^2)$. In particular, we can associate to any circle, $C_i$ two, possibly equal closed subsets of the plane, a simply connected domain bounded by $C_i$, denoted $D(C_i)$, and the closure of the connected component of $\RR^2\setminus\kappa$ that non-trivially intersects $D(C_i)$ and has $C_i$ as a boundary component, denoted $R(C_i)$.
	
	\begin{definition}[The abstract tree of configurations]\label{dfn:abstracttree}
		Let $\kappa=(C_1,\dots,C_n)$ be an element of $\Conf_n(S^1,\RR^2)$. Define the \emph{abstract labeled tree} of $\kappa$, we denote this $T(\kappa)$, to have a vertex $i$ for each $i\in\{1,\dots,n\}$ and a directed edge from vertex $i$ to vertex $j$ for $i\neq j$, labeled $(i,j)$, if and only if $C_j$ is a boundary component of $R(C_i)$; additionally, there is an unlabeled vertex, $\varnothing$, and an edge $(\varnothing,i)$ for each $C_i$ that does not lie in $D(C_j)$ for all $j\neq i$.
		
		Let $\eta:=\{C_1,\dots,C_n\}$ be an unlabeled configuration of $n$ circles in the plane. We may give this configuration an arbitrary labeling to yield $\widetilde{\eta}=(C_{i_1},\dots,C_{i_n})$ an element of $\Conf_n(S^1,\RR^2)$. Construct $T(\widetilde{\eta})$ and forget the labeling of the vertices to obtain $T(\eta)$, which we call the \emph{abstract unlabeled tree} of $\eta$.
		
		Two abstract labeled trees $T(\kappa)$ and $T(\kappa')$ are equal if and only if $T(\kappa)$ and $T(\kappa')$ have the same vertex and edge sets. Two abstract trees $T(\eta)$ and $T(\eta')$ are equal if and only if there exists an arbitrary labeling of $\eta$ and $\eta'$ such that $T(\widetilde{\eta})$ and $T(\widetilde{\eta}')$ are equal.
	\end{definition}
	
	\begin{remark}
		We note that planar trees are defined by inequalities of the coordinates of circles in a configuration, whereas abstract trees are defined by the sets bounded by the circles in a configuration. The proof of Lemma~\ref{lem:abstracttreeandtreerelate} will show that the abstract tree of a configuration is an isomorphism class of planar trees.
	\end{remark}
	
	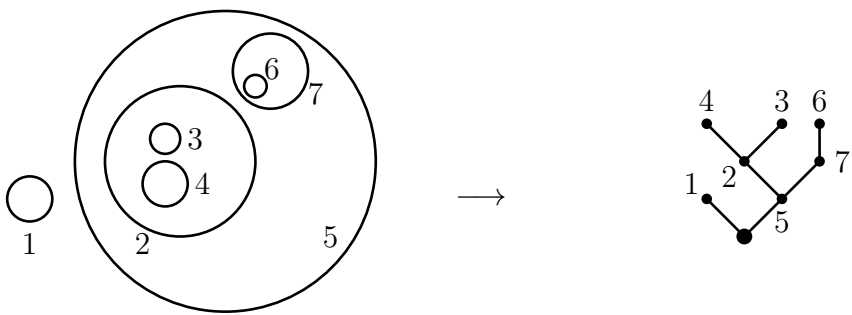
\begin{figure}[htb]
		\centering
		\begin{tikzpicture}[line width=1pt]
			
			\draw (-2,-0.5) circle (0.3cm);
			\draw (0,0) circle (1cm);
			\draw (-0.2,0.3) circle (0.2cm);
			\draw (-0.2,-0.3) circle (0.3cm);
			\draw (0.6,0) circle (2cm);
			\draw (1,1) circle (0.15cm);
			\draw (1.2,1.2) circle (0.5cm);
			
			\node at (-2,-1.1) {$1$};
			\node at (-0.5,-1.1) {$2$};
			\node at (0.2,0.3) {$3$};
			\node at (0.3,-0.3) {$4$};
			\node at (2,-1) {$5$};
			\node at (1.22,1.23) {$6$};
			\node at (1.8,0.9) {$7$};
			
			\node at (4,-0.5) {$\longrightarrow$};
			
			\begin{scope}[xshift=7cm,yshift=-0.5cm]
				\draw (0,0) -- (0.5,-0.5) -- (1.5,0.5) -- (1.5,1)   (1,0) -- (0,1)   (0.5,0.5) -- (1,1);
				
				\filldraw (0,0) circle (1.5pt);
				\filldraw (0.5,-0.5) circle (2.5pt);
				\filldraw (1.5,0.5) circle (1.5pt);
				\filldraw (1,0) circle (1.5pt);
				\filldraw (0,1) circle (1.5pt);
				\filldraw (0.5,0.5) circle (1.5pt);
				\filldraw (1,1) circle (1.5pt);
				\filldraw (1.5,1) circle (1.5pt);
				
				\node at (-0.2,0.2) {$1$};
				\node at (0.3,0.3) {$2$};
				\node at (1,1.3) {$3$};
				\node at (0,1.3) {$4$};
				\node at (1,-0.3) {$5$};
				\node at (1.5,1.3) {$6$};
				\node at (1.8,0.5) {$7$};
			\end{scope}
			
		\end{tikzpicture}
		\caption{A labeled configuration $\kappa$ of circles in $\Conf_7(S^1,\RR^2)$, and the associated abstract labeled tree $T(\kappa)$. (Reprinted from \cite{Curry2026})}
		\label{fig:config_tree}
	\end{figure}
	
	\begin{lemma}\label{lem:abstracttreeandtreerelate}
		The abstract labeled trees of $\Conf_n(S^1,\RR^2)$ index the isomorphism classes of planar labeled trees of $\Conf_n(S^1,\RR^2)$. Similarly, the abstract trees of $\UConf_n(S^1,\RR^2)$ index the isomorphism classes of planar trees of $\UConf_n(S^1,\RR^2)$.
	\end{lemma}
	\begin{proof}
		Let $\kappa=(C_1,\dots,C_n)$ be an element of $\Conf_n(S^1,\RR^2)$. Let $\overline{T}_{\kappa}$ forget the planar ordering of the children of vertices of the planar labeled tree of $\kappa$, $T_{\kappa}$. 
		Note that if two planar trees are isomorphic then forgetting their planar ordering yields equal trees.
		We will show $\overline{T}_{\kappa}=T(\kappa)$. Since both $\overline{T}_{\kappa}$ and $T(\kappa)$ have $n$ non-root vertices labeled $1$ through $n$, we only need to show that they have the same directed edge set. 
		This is true because $C_j$ is nested inside $C_i$ if and only if $C_j$ lies in $D(C_i)$, and the non-existence of a $C_k$ that is simultaneously nested by $C_i$ with $C_j$ nested by $C_k$ is equivalent to $C_j$ being a component of the boundary of $R(C_i)$.
		Similarly, $C_i$ is not nested inside any other $C_k$ if and only if $C_i$ lies in the boundary of the region of the plane that is not bounded by any circle.
		
		Let $\eta=\{C_1,\dots,C_n\}$ be an element of $\UConf_n(S^1,\RR^2)$. Giving $\eta$ an arbitrary labeling, $\widetilde{\eta}=(C_{i_1},\dots,C_{i_n})$, corresponds to a labeling of $T_{\eta}$, denote it $\widetilde{T}_{\eta}$, such that $\widetilde{T}_{\eta}$ is the labeled planar tree of $\widetilde{\eta}$. By the previous argument, $\overline{T}_{\widetilde{\eta}}=T(\widetilde{\eta})$ and forgetting their labeling yields $T_{\eta}=T(\eta)$. Hence, the latter statement of the lemma is true.
	\end{proof}
	
	We will see in Section~\ref{sec:rounding} that the connected components of $\J_n(X)$ and $\UJ_n(X)$ are indexed by the same labeled and unlabeled trees that witness the containment order of the curves in a Jordan configuration.
	
	\begin{definition}[The abstract tree of Jordan configurations]\label{dfn:abstracttreeJrdancurves}
		Let $f=([f_1],\dots,[f_n])$ be a labeled Jordan configuration in $X$. Define the \emph{abstract labeled tree} of $f$, called $T(f)$, to have a vertex $i$ for each $f_i$ and a directed edge from vertex $i$ to vertex $j$, written as $(i,j)$, if and only if $f_j(S^1)$ is a component of the boundary of $R(f_i)$; additionally, there is an unlabeled vertex, $\varnothing$, for the region of $X\setminus f$ that is not $D(f_i)$ for some $i$ and an edge $(\varnothing,i)$ for each $i$ such that the image of $f_i(S^1)$ lies in the boundary of this region.
		Since the abstract tree depends on the image of a Jordan configuration and not on the parametrization of the curve, it is well defined.
		
		For $g:=\{[g_1],\dots,[g_n]\}$ an element of $\UJ_n(X)$, we may give this Jordan configuration an arbitrary labeling to yield $\widetilde{g}=([g_{i_1}],\dots,[g_{i_n}])$, an element of $\J_n(X)$. Construct $T(\widetilde{g})$ and forget the labeling of the vertices to obtain $T(g)$, which we call the \emph{abstract unlabeled tree} of $g$. 
		
		Two abstract labeled trees of Jordan configurations, $T(f)$ and $T(f')$, are equal if and only if $T(f)$ and $T(f')$ have equal vertex and edge sets. Two abstract unlabeled trees of Jordan configurations, $T(g)$ and $T(g')$, are equal if and only if there exists an arbitrary labeling of $g$ and $g'$ such that $T(\widetilde{g})$ and $T(\widetilde{g}')$ are equal. Two arbitrary labelings of a Jordan configuration result in two abstract labeled trees that differ by a permutation of the labels; hence the structure of the abstract unlabeled tree does not depend on the arbitrary labeling.
	\end{definition}
	
	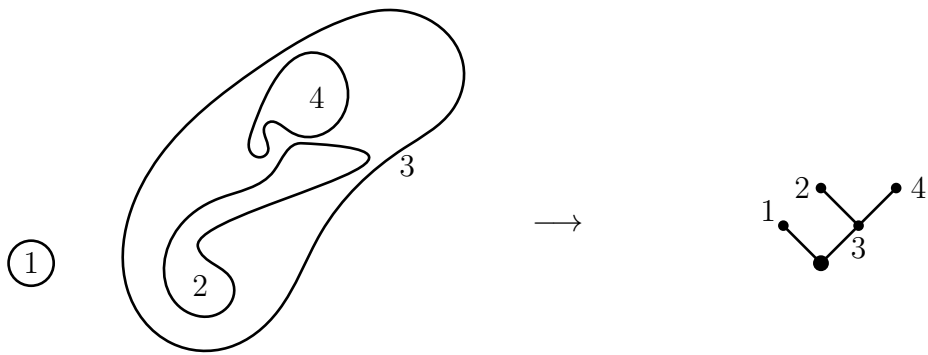
\begin{figure}[htb]
		\centering
		\begin{tikzpicture}[line width=1pt]
			
			\draw (-3,-1) circle (0.3cm);
			
			\draw[use Hobby shortcut,closed=true, scale=1.3, rotate=55]
			(-1.8,-.2) .. (-1.6,.5) .. (-1,.9) .. 
			(-.3,1) .. (1,.7) .. (2,.2) .. (2.22,0) ..
			(2.42,-.45) .. (2.25,-1) ..
			(1.8,-1.18) .. (1.2,-1) ..
			(0,-.8) ..
			(-.9,-.9) .. (-1.5,-.7);
			\draw[use Hobby shortcut,closed=true, scale=1.3, rotate=45,xshift=-10]
			(.6,0) .. (0,.2) .. (-1,-.5) .. (-.5,-.5) .. (-.5,0) .. (1,0);
			\draw[use Hobby shortcut,closed=true, scale=1, rotate=180, yshift=-50, xshift=-55]
			(1.8,0.9) .. (1.8,1.3) .. (2,1.3) .. (2,1) .. (1,0) .. (.8,.2) .. (1.5,1);
			
			\node at (-3,-1) {$1$};
			\node at (-0.75,-1.3) {$2$};
			\node at (2,.3) {$3$};
			\node at (.8,1.2) {$4$};
			
			\node at (4,-0.5) {$\longrightarrow$};
			
			\begin{scope}[xshift=7cm,yshift=-0.5cm]
				\draw (0,0) -- (0.5,-0.5) -- (1.5,0.5)   (1,0) -- (0.5,0.5);
				
				\filldraw (0,0) circle (1.5pt);
				\filldraw (0.5,-0.5) circle (2.5pt);
				\filldraw (1.5,0.5) circle (1.5pt);
				\filldraw (1,0) circle (1.5pt);
				\filldraw (0.5,0.5) circle (1.5pt);
				
				\node at (-0.2,0.2) {$1$};
				\node at (0.25,0.5) {$2$};
				\node at (1,-0.3) {$3$};
				\node at (1.8,0.5) {$4$};
			\end{scope}
			
		\end{tikzpicture}
		\caption{A labeled Jordan configuration $f$ in $\J_4(\RR^2)$, and the associated abstract labeled tree $T(f)$.}
		\label{fig:multicurv_tree}
	\end{figure}
	
	\begin{lemma}
	The maps sending abstract trees of configurations to abstract trees of Jordan configurations induced by the embeddings $\Conf_n(S^1,\RR^2)\to \J_n(\RR^2)$ and $\UConf_n(S^1,\RR^2)\to \J_n(\RR^2)$ are the identity.
	\end{lemma}
	\begin{proof}
		This is true due to Definitions~\ref{dfn:abstracttree} and~\ref{dfn:abstracttreeJrdancurves} agreeing under the embedding in Theorem~\ref{thm:embeddingconfspaces}.
	\end{proof}
	
	\begin{lemma}\label{lem:pathconnembsgivethesameabstracttree}
		If there exists a path from a labeled or unlabeled Jordan configuration $f$ to another $f'$, then the corresponding abstract labeled or unlabeled trees $T(f)$ and $T(f')$ are equal. 
	\end{lemma}
	\begin{proof}
		It is clear that the vertex sets are equal and we need to show that the edge sets are too. Let $p_t$ be a path from $f=([f_1],\dots,[f_n])$ to $f'=([f_1'],\dots,[f_n'])$. Suppose $D(f_j)$ lies in the disk $D(f_i)$. For all $t$, we have $n$ labeled Jordan configurations $p_t$. This means for each $t$, the $j^{th}$ component of $p_t$ must lie in the disk with boundary the $i^{th}$ component of $p_t$. Therefore, $D(f'_j)$ lies in the disk $D(f'_i)$ and the edge sets of $T(f)$ and $T(f')$ are equal.
		
		Lifting a path between two elements $g$ and $g'$ in $\UJ_{n}(X)$ by giving them a labeling that respects the path yields the unlabeled case.
	\end{proof}
	
	
	\section{Rounding Jordan configurations}\label{sec:rounding}
	
	
	Here we round all Jordan configurations in our surface to be circles in the sense of curves equidistant to a point in $X$, Theorem~\ref{thrm:roundingcurvesinsurfaces}. This combined with Proposition~\ref{prop:components} allows us to index the connected components of $\UJ_n(X)$ and focus on the finite dimensional subspace $\RUJ_n(X)$. To round Jordan configurations in surfaces we lift to the universal cover of constant curvature and round each curve relative to the tangent plane of the surface at the center point of the curve, see Figure~\ref{fig:domain_surface_tangentplane}. Thus, we focus on rounding Jordan configurations in the plane, Theorem~\ref{thrm:preroundingJordanconfigurationsintheplane}. An added result due to the equivariant centers of Belegradek and Ghomi in Theorems~\ref{thrm:extequcenter}~and~\ref{thrm:extequcenterplane} is that these strong deformation retractions become equivariant. In fact, Belegradek and Ghomi prove Theorem~\ref{thrm:roundingcurvesinsurfaces} for the case $n=1$. They even prove that $\mathcal{D}(S^2)$ equivariantly strong deformation retracts onto the subspace of round Jordan domains.
	
	\begin{remark}
		It should be stated that the deformation retractions can be performed by curve shortening flow on rectifiable curves \cite{lauer2013new}.
	\end{remark}

	\begin{figure}[htb]
		\centering
		\begin{tikzpicture}[line width=1pt]
			
			\begin{scope}[rotate=15,yshift=45,xshift=10]
				\draw (-4,1.8) -- (-2,0.5);
				\draw (0,1.8) -- (2,0.5);
				\draw (-4,1.8) -- (0,1.8);
				\draw (-2,0.5) -- (2,0.5);
			\end{scope}
			
			\draw[use Hobby shortcut,closed=true, scale=1.29,xshift=2] (-1.6,.8) .. (-1.3,.91) .. (-1,.9) .. (-.5,.9) .. (-.5,.6) .. (-1,.5);
			
			\draw[use Hobby shortcut,closed=false, scale=1.3]
			(2.4,-1) .. (.5,0.1) .. (-1.15,-.5);
			
			\draw[use Hobby shortcut,closed=false, scale=1.3]
			(-1.15,-.5) .. (-1.75,-.2) .. (-3.5,-1);
			
			\draw[use Hobby shortcut,closed=false, scale=1.3]
			(2.4,-1) .. (-.9,1) .. (-3.5,-1);
			
			\filldraw (-1.3,1) circle (1.5pt);
			\node at (-1.6,.9) {$p$};
			
			\filldraw (-1.3,2.6) circle (1.5pt);
			\node at (-1.6,2.5) {$p$};

			\draw[use Hobby shortcut,closed=true, scale=1.29,xshift=6,yshift=45, rotate=15] (-1.6,.8) .. (-1.3,.91) .. (-1,.9) .. (-.5,.9) .. (-.5,.6) .. (-1,.5);
			
			\node at (-6,2.5) {$T_pX:$};
			\node at (-6,0) {$X:$};
			\node at (-2.9,2.3) {$\exp_{p}^{-1}([f])$};
			\node at (-2.4,.6) {$[f]$};
			
		\end{tikzpicture}
		\caption{The inverse exponential function in Definition~\ref{def:exp} of $X$ at $p$ lifts a curve $[f]$ in $\UJ_{n}(X)$ to the tangent plane of $X$ at $p$, where points/vectors in the tangent plane correspond to geodesics in $X$.}
		\label{fig:domain_surface_tangentplane}
	\end{figure}
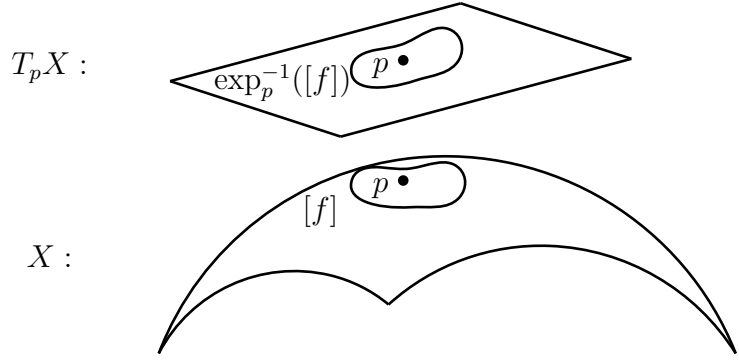
	
	The strong deformation retractions in this section depend on the containment order of the Jordan configurations. 
	Given a Jordan configuration, $f=\{f_1,\dots,f_n\}$, its abstract tree, $T(f)$, defines a strict partial ordering of $\{\varnothing,1,\dots,n\}$: we write $j\prec i$ if $D(f_j)\subsetneq D(f_i)$. 
	Let $d_j$ be the number of $i\neq\varnothing$ such that $j\prec i$. 
	Take $d$ to be the maximum $d_j$ for $j\in \{1,\dots,n\}$.
	
	\subsection{Convex curves in the plane}\label{sec:convex}
	
	An intermediate step before proving Theorem~\ref{thrm:preroundingJordanconfigurationsintheplane} is to show that the space of convex Jordan configurations strong deformation retracts onto the space of round Jordan configurations.
	
	\begin{definition}[Convex Jordan configurations]\label{dfn:cnvxjrdncrves}
		We say a curve $[f]\in\J_1(\RR^2)$ is \emph{convex} if $D(f)$ is a convex region of the plane. Let $\CJ_n(\RR^2)$ denote the subspace of $\J_n(\RR^2)$ where the entries of the tuples are convex. Similarly, let $\CUJ_n(\RR^2)$ denote the subspace of $\UJ_n(\RR^2)$ where elements of the sets are convex.
	\end{definition}
	
	Convex Jordan configurations are a closed $\Con(\RR^2)$-invariant subspace and there are a number of known $\Con(\RR^2)$-equivariant centers defined on them including the centroid. Note that the only choice of equivariant center on $\RUJ_n(\RR^2)$ with respect to conformal transformations of the plane is the canonical choice of centers of circles.
	Centroids of convex domains are known to lie in their interior, be equivariant under conformal maps of the plane, and the map that computes the centroid of domains is known to be continuous \cite{bonnesen1987theory,moszynska2006selected}.
	There are other choices for $\Con(\RR^2)$-equivariant centers on convex curves, for example the Steiner point \cite{schneider2013convex,moszynska2006selected}.
	
	Given $f=\{[f_1],\dots,[f_n]\}$, an element of $\CUJ_n(\RR^2)$, let $R_{f_i}$ be the maximum distance from $f_i(S^1)$ to the center of $[f_i]$ provided by Corollary~\ref{cor:existcenter}, written as $c(D(f_i))=c_{f_i}$, and let $r_{f_i}$ be the minimum distance between $f_i(S^1)$ and $c_{f_i}$. 
	
	\begin{theorem}\label{thrm:convdefretractontojrdncurves}
		The spaces $\CJ_n(\RR^2)$ and $\CUJ_n(\RR^2)$ strong deformation retract onto $\RJ_n(\RR^2)$ and $\RUJ_n(\RR^2)$, respectively, in a way that is equivariant under $\Con(\RR^2)$. 
	\end{theorem}
	\begin{proof}
		Let $f=\{[f_1],\dots,[f_n]\}$ be an element of $\CUJ_n(\RR^2)$. 
		We will construct a piecewise linear homotopy.
		For $j\in \{1,\dots,n\}$, define
		$$F_{f_j,d_j+1}(t,x):=f_j(x)+t\left(\frac{c_{f_j}-f_j(x)}{|c_{f_j}-f_j(x)|}\right)(|c_{f_j}-f_j(x)|-r_{f_j}),$$
		which linearly contracts the image of $f_j$ to the circle centered at $c_{f_j}$ with radius $r_{f_j}$, so the image of $F_{f_j,d_j+1}$ is invariant under precomposition with a homeomorphism of a circle and passing to the equivalence class $[F_{f_j,d_j+1}(t,-)]$ is well defined. 
		For $1\leq k\leq d_j$, let $c_{f_j,k}$ be the center of the circle $F_{f_j,d_j+1-(k-1)}(1,S^1)$ and let $j_k\in\{1,\dots,n\}$ be the unique index such that there exists exactly $k-1$ distinct $l$ such that $j\prec l\prec j_k$.
		Define $F_{f_j,d_j+1-k}(t,x)$ to be
		$$\left(c_{f_j,k}+t\left(\frac{c_{f_{j_k}}-c_{f_j,k}}{R_{f_{j_k}}}\right)(R_{f_{j_k}}-r_{f_{j_k}})\right)+\left(F_{f_j,d_j+1-(k-1)}(1,x)-c_{f_j,k}\right)\left(1-t\left(\frac{R_{f_{j_k}}-r_{f_{j_k}}}{R_{f_{j_k}}}\right)\right),$$
		and for $d_j-d\leq k\leq -1$ define $$F_{f_j,d_j+1-k}(t,x):=f_j(x).$$
		We have that $F_{f_j,k}$ is stationary for $k>d_j+1$ and linearly translates and scales the circle $F_{f_j,k+1}(1,S^1)$ for $k<d_j+1$, so $F_{f_j,k}(t,x)$ is invariant under precomposition with a homeomorphism of a circle and passing to the equivalence class $[F_{f_j,k}(t,-)]$ is well defined.
		
		We will show $\{[F_{f_1}(t,-)],\dots,[F_{f_n}(t,-)]\}$ is a path of Jordan configurations from $\{[f_1],\dots,[f_n]\}$ to a round Jordan configuration defined by concatenation
		$$F_{f_j}(t,-):=(F_{f_j,d+1}* F_{f_j,d}*\cdots *F_{f_j,d_j}*\cdots* F_{f_j,2}* F_{f_j,1})(t,-)$$
		for all $j$. Note that when $\{[f_1],\dots,[f_n]\}$ is an element of $\RUJ_n(\RR^2)$ we have $|c_{f_j}-f_j(x)|-r_{f_j}=0$, and $r_{f_j}=R_{f_j}$ for all $j$, so $\{[F_{f_1}(t,-)],\dots,[F_{f_n}(t,-)]\}$ leaves a round Jordan configuration stationary. 
		The concatenation is well defined since each $F_{f_j,k}(t,-)$ starts at $F_{f_j,k+1}(1,-)$ and ends at $F_{f_j,k-1}(0,-)$. 
		Since $R_{f_j}$, $r_{f_j}$, and $c_{f_j}$ all change continuously as $f_j$ changes continuously, it is sufficient to show $\{[F_{f_1,k}(t,-)],\dots,[F_{f_n,k}(t,-)]\}$ is a path of Jordan configurations, that is, each pair of curves stays disjoint for all $k$ and $t$. 
		
		Suppose $\{[F_{f_1,d_i+1}(0,-)],\dots,[F_{f_n,d_i+1}(0,-)]\}$ is a Jordan configuration for some $i$, this is true for $d_i=d$ since $\{[F_{f_1,d+1}(0,-)],\dots,[F_{f_n,d+1}(0,-)]\}=\{[f_1],\dots,[f_n]\}$.
		We have $\{[F_{f_j,d_i+1}(t,-)]\mid j\prec i\}$ is a linear path of round Jordan configurations in the interior of $f_i$ that contracts the configuration $\{[F_{f_j,d_i+1}(0,-)]\mid j\prec i\}$, centered at $c_{f_i}$, by a factor of $\frac{r_{f_i}}{R_{f_i}}$. That is $\{[F_{f_j,d_i+1}(0,-)-c_{f_i}]\mid j\prec i\}$ and $\{[F_{f_j,d_i+1}(1,-)-c_{f_i}]\mid j\prec i\}$ are the same configuration shrunken by a factor of $\frac{r_{f_i}}{R_{f_i}}$, and therefore $\{[F_{f_j,d_i+1}(t,-)]\mid j\prec i\}$ is a configuration for all $t$. Note that $[F_{f_i,d_i+1}(0,-)]$ bounds and is disjoint from $\{[F_{f_j,d_i+1}(0,-)]\mid j\prec i\}$. Further, the distance of $F_{f_i,d_i+1}(t,x)$ from $c_{f_i}$ decreases by a factor of at most $\frac{r_{f_i}}{R_{f_i}}$ for all $x\in S^1$, whereas the distance of $F_{f_j,d_i+1}(t,x)$ from $c_{f_i}$ decreases by a factor of exactly $\frac{r_{f_i}}{R_{f_i}}$ for all $x\in S^1$ and $j\prec i$, so $[F_{f_i,d_i+1}]$ stays disjoint from $\{[F_{f_j,d_i+1}(t,-)]\mid j\prec i\}$. 
		
		The paths $[F_{f_i,d_i+1}(t,-)]$ and $\{[F_{f_j,d_i+1}(t,-)]\mid j\prec i\}$ are contained in $D(f_i)$, the same is said for all $l$ with $d_l=d_i$ we have $[F_{f_l,d_i+1}(t,-)]$ and $\{[F_{f_j,d_i+1}(t,-)]\mid j\prec l\}$ are contained in $D(f_l)$ for all $t$. If $d_i> d_l$, then $[F_{f_l,d_i+1}(t,-)]$ is stationary. It follows that $[F_{f_i,d_i+1}(t,-)]$ and $\{[F_{f_j,d_i+1}(t,-)]\mid j\prec i\}$ are disjoint from $[F_{f_l,d_i+1}(t,-)]$ and $\{[F_{f_j,d_i+1}(t,-)]\mid j\prec l\}$ for $l$ with $d_l=d_i$, and $[F_{f_i,d_i+1}(t,-)]$ and $\{[F_{f_j,d_i+1}(t,-)]\mid j\prec i\}$ are disjoint from $[F_{f_l,d_i+1}(t,-)]$ for $l$ with $d_i> d_l$. By iterating this procedure from $d_i=d$ to $d_i=1$, this shows $\{[F_{f_1}(t,-)],\dots,[F_{f_n}(t,-)]\}$ is a path of a convex Jordan configuration to a round Jordan configuration, constant if $\{[f_1],\dots,[f_n]\}$ was already a round Jordan configuration.
		
		Lastly, we note that this strong deformation retraction is $\Con(\RR^2)$-equivariant. Applying a conformal map to $f$ does not affect the nesting, so the abstract tree of $f$ stays the same after applying a conformal map, and hence the partial ordering defined by the abstract tree does not change. Checking equivariance is similar to Belegradek and Ghomi's check in \cite[Prop.~4.1]{belegradek2025point} which follows from equivariance of our chosen centers. The difference in our deformation retraction is checking if the scaling is equivariant, however since distance ratios are preserved by conformal linear transformations and translations, $$\frac{r_{f_j}}{|c_{f_j}-f_j(x)|}=\frac{g(r_{f_j})}{|g(c_{f_j})-g(f_j(x))|}=\frac{r_{g(f_j)}}{|c_{g(f_j)}-g(f_j(x))|}$$ and $$\frac{r_{f_{j_k}}}{R_{f_{j_k}}}=\frac{g(r_{f_{j_k}})}{g(R_{f_{j_k}})}=\frac{r_{g(f_{j_k}})}{R_{g(f_{j_k})}}.$$
		
		This construction of a $\Con(\RR^2)$-equivariant strong deformation retraction does not depend on a labeling, so the theorem holds.
	\end{proof}
	
	See Figure~\ref{fig:convex_def} for an example of the above deformation retraction applied to a convex Jordan configuration of 3 curves in the plane.
	
		\begin{figure}[htb]
		\centering
		\begin{tikzpicture}[line width=1pt,scale=.8]
			
			\begin{scope}[xshift=-8.5cm]
			\draw (-4,0) ellipse (3cm and 1cm);
			\draw (-2,0) ellipse (.25cm and .5cm);
			\draw[dotted] (-2,0) circle (.25cm);
			\draw (-5,0) circle (.75cm);
			
			\draw[thick,fill] (-5,0) circle (.05cm);
			
			\draw[thick,fill] (-2,0) circle (.05cm);
			
			\node at (0.1,0) {$\longrightarrow$};
			\end{scope}
			
			\draw (-4,0) ellipse (3cm and 1cm);
			\draw[dotted] (-4,0) circle (1cm);
			\draw[dashed] (-4,0) circle (3cm);
			\draw (-2,0) ellipse (.25cm and .25cm);
			\draw (-5,0) circle (.75cm);
			
			\draw[thick,fill] (-4,0) circle (.05cm);
			\node at (-3.85,-.15) {$c$};
			\draw[dash dot] (-4,0) -- (-3.5,0.866025403784439);
			\node at (-3.5,.3) {$r$};
			\draw[dash dot dot] (-4,0) -- (-3.5,2.95803989154980802);
			\node at (-3.3,1.7) {$R$};
			
			\node at (0.1,0) {$\longrightarrow$};
			
			\draw (2,0) circle (1cm);
			\draw (2.66666667,0) circle (.0833333cm);
			\draw (1.66666667,0) circle (.25cm);

		\end{tikzpicture}
		\caption{An element of $\CUJ_3(\RR^2)$ and its deformation to a round Jordan configuration.
		First the nested curves contract to the largest circle contained in each curve centered at its center (dotted) and leaves circles stationary. 
		Then the deformation contracts the un-nested curve to the largest circle it contains centered at its center $c$ (dotted), and scales the configuration nested in the un-nested curve by a factor of $r/R$ relative to $c$, so that it becomes the same configuration in the dotted circle as it was in the smallest circle containing the un-nested curve centered at $c$ (dashed).}
		\label{fig:convex_def}
	\end{figure}
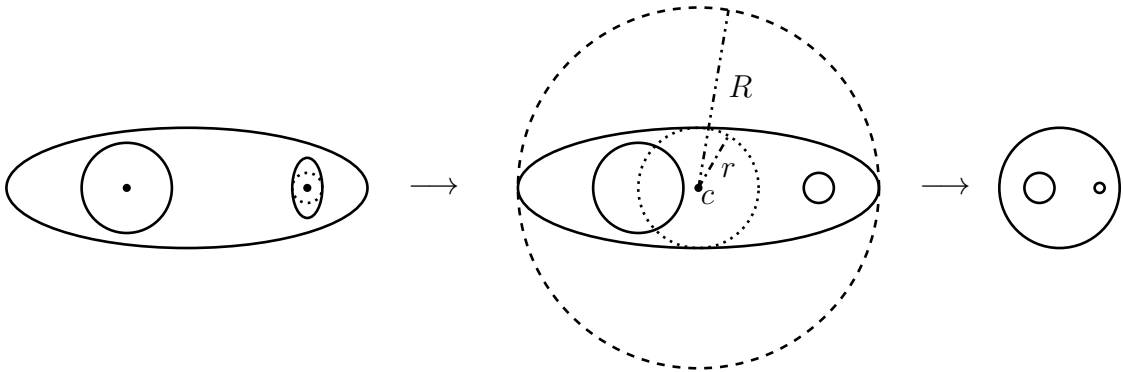
	
	The construction of the strong deformation retraction in the previous proof behaves well with respect to parametrizations of the curves so a similar statement can be said for the labeled and unlabeled embedding space of circles in the plane.

	
	\subsection{Jordan configurations in the plane}\label{sec:proof}
	
	In this subsection we prove Theorem~\ref{thrm:roundingcurvesinsurfaces} in the case of the plane which is a generalization of Belegradek and Ghomi's application of Theorem~\ref{thrm:extequcenterplane} \cite[Prop.~4.1]{belegradek2025point}. 
	Here we take their construction of a  $\Con(\RR^2)$-equivariant deformation retraction of the space of Jordan domains onto the space of round Jordan domains in the plane, and we update it to $\Con(\RR^2)$-equivariantly deformation retract the space of Jordan configurations in the plane onto the space of round Jordan configurations in the plane.
	
	Recall for every simple closed curve in the plane, $[f]\in \J_1(\RR^2)$, we have $D(f)\in\mathcal{D}(\RR^2)$ to denote the union of the interior of the image of $f$ and the image of $f$. By Corollary~\ref{cor:existcenter}, there is a continuous $\Con(\RR^2)$-equivariant choice of interior point of $D(f)$, written as $c(D(f))$, such that if $D(f)$ is round then $c(D(f))$ is the usual center of a circle. Also recall the strict partial ordering of the indices of Jordan configurations given by their abstract tree, as well as the associated numbers $d_j$, the number of $i\neq\varnothing$ such that $j\prec i$, and $d$, the maximal such $d_j$. 
	
	Let $f=\{[f_1],\dots,[f_n]\}$ be an element of $\UJ_n(\RR^2)$.  
	If $d=0$, then $f_i$ does not lie in $D(f_j)$ for all $i\neq j$, it follows from Lemma~\ref{lem:conformalmapping} that there is a continuous map from $\UJ_n(\RR^2)$ to $\UEmb_n(B^2,\RR^2)$ induced by sending each curve, $[f_i]$, to the unique map $\gamma_{D(f_i),c(D(f_i)),(1,0)}$. In particular, there is a continuous map sending \[f^{(0)}:=\{[f_i]\mid d_i=0\}\] to $\{\gamma_{D(f_i),c(D(f_i)),(1,0)}\mid d_i=0\}$ in $\UEmb_m(B^2,\RR^2)$, where $\gamma_{D(f_i),c(D(f_i)),(1,0)}$ is conformal on the interior of $B^2$ and $m=|\{i\mid d_i=0\}|$. Write $\gamma_{f_i}$ for $\gamma_{D(f_i),c(D(f_i)),(1,0)}$.
	
	The idea is to use the conformal parametrization of the interior of the un-nested curves, $f^{(0)}$, to move the nested curves while rounding the un-nested curves. To ensure that two un-nested curves are disjoint while rounding them, we shrink them so that the rounding happens in the domain of the original curve. For $y\in[0,1)$, let $\phi_y:[0,1]\to(B^2\to B^2)$ be the path in $\Emb(B^2,B^2)$ sending $s$ to the map $x\mapsto x-ysx$. Note that $\lim\limits_{y\to1}\phi_y(1)$ is the constant map at the center of $B^2$, denoted $o$; and note for $y<y'$ we have $\left(\phi_{y'}(s)\right)\left(B^2\right)\subsetneq\left(\phi_{y}(s)\right)\left(B^2\right)$. We define $h_{t,f_i,y}:B^2\to\RR^2$ to be the rounding deformation retraction of an un-nested curve, $f_i$, and its interior after $\phi_y$, that is
	$$h_{t,{f_i},y}(x):=\begin{cases}
		c(D(f_i))+\frac{\gamma_{f_i}\left(t(\phi_y(1))(x)\right)-c(D(f_i))}{t}&t\in(0,1]\\
		c(D(f_i))+(1-y)(d\gamma_{f_i})_o(x)&t=0
	\end{cases},$$ 
	which is continuous in $t$,
	\begin{align*}
		\lim\limits_{t\to0}h_{t,{f_i},y}=&\lim\limits_{t\to0}c(D(f_i))+\frac{\gamma_{f_i}\left(t(\phi_y(1))(x)\right)-c(D(f_i))}{t}\\
		=&c(D(f_i))+\lim\limits_{t\to0}\frac{\gamma_{f_i}\left(t(1-y)(x)\right)-c(D(f_i))}{t}\cdot\frac{(1-y)t}{(1-y)t}\\
		=&c(D(f_i))+\lim\limits_{t\to0}\frac{\gamma_{f_i}\left(t(1-y)(x)\right)-c(D(f_i))}{(1-y)t}\cdot\frac{(1-y)t}{t}\\
		=&c(D(f_i))+\lim\limits_{u\to0}\frac{\gamma_{f_i}\left(ux\right)-c(D(f_i))}{u}\cdot(1-y)&u\coloneqq t(1-y)\\
		=&c(D(f_i))+(d\gamma_{f_i})_o(x)\cdot(1-y).
	\end{align*}
	Calling this a `rounding' deformation of a Jordan configuration is a misnomer since linear maps, such as $h_{0,f_i,y}$, do not generally preserve circles, but they do send circles to ellipses. In particular, the set $h_{0,{f_i},y}(B^2)$ is convex since linear maps send convex sets to convex sets.
	
	Let $r_{f_i}$ be the radius of the circle inscribed in $D(f_i)$ with center $c(D(f_i))$, that is, $r_{f_i}$ is the infimum of $d_{\RR^2}(c(D(f_i)),f_i(x))$ for $x\in S^1$. Denote by $M_{{f_i},y}$ the supremum of 
	$$d_{\RR^2}\left(h_{t,{f_i},y}\left(x\right),c(D(f_i))\right),$$
	for $x\in B^2$ and define the map $\varphi_{{f_i}}:[0,1)\to(0,\infty)$ by $y\mapsto M_{{f_i},y}$.
	
	\begin{lemma}
		The value $\varphi_{{f_i}}^{-1}(r_{f_i})$ is a  continuous choice of $y\in[0,1)$ such that $h_{t,{f_i},y}(B^2)$ lies in $D(f_i)$ for all $t$. 
	\end{lemma}
	\begin{proof}
		The image of $h_{t,{f_i},y}$ with $y$ fixed is compact and therefore a closed and bounded subset of $\RR^2$.
		Note that $\varphi_{{f_i}}$ is strictly decreasing and continuous with $\lim\limits_{y\to1}M_{{f_i},y}=0$ since $\lim\limits_{y\to1}\phi_y$ is the constant map. Hence, $\varphi_{f_i}$ is an embedding that decreases to $0$ as $y\to1$. 
		Since $c$ and $h_{t,{f_i},y}$ are continuous with respect to $D(f_i)$, the map $[f_i]\mapsto\varphi_{f_i}$ is continuous. There is a $y$ such that $M_{{f_i},y}=r_{f_i}$ since $0<r_{f_i}\leq M_{{f_i},0}$. Thus, the maps $[f_i]\mapsto \varphi^{-1}_{f_i}$ and $[f_i]\mapsto r_{f_i}$ are continuous and we obtain a continuous function $[f_i]\mapsto \varphi^{-1}_{f_i}(r_{f_i})$.
	\end{proof}
	
	For $i$ such that $d_i=0$, precomposing $\gamma_{f_i}$ with the image of $\phi_{\varphi^{-1}_{f_i}(r_{f_i})}$ yields the homotopy $\Gamma_0:[0,1]\times \UJ_m(\RR^2)\to\UEmb_m(B^2,\RR^2)$, defined as
	$$\Gamma_0(s,f^{(0)})\coloneqq\{\gamma_{f_i}\circ(\phi_{\varphi^{-1}_{f_i}(r_{f_i})}(s))\mid d_i=0\}.$$ 
	Then $\Gamma_0$ shrinks all un-nested curves and their interiors so that concatenation of the homotopy $\Gamma_0':[0,1]\times \UJ_m(\RR^2)\to\UEmb_m(B^2,\RR^2)$, defined by 
	$$\Gamma_0'\left(t,f^{(0)}\right)\coloneqq\{h_{t,{f_i},\varphi^{-1}_{f_i}(r_{f_i})}\mid d_i=0\},$$ 
	is well defined and lies within $\bigcup\limits_{d_i=0}D(f_i)$.
	
	For all $j\preceq i$ such that $d_i=0$, define the map $\psi_{f_j}\coloneqq(\gamma_{f_i}^{-1}\circ\gamma_{f_j})$, so
	$$\gamma_{f_j}=\gamma_{f_i}\circ\psi_{f_j}=\gamma_{f_i}\circ\left(\phi_{\varphi^{-1}_{f_i}(r_{f_i})}(0)\right)\circ\psi_{f_j}.$$
	We get a homotopy $F_0:[0,1]\times \UJ_{n}(\RR^2)\to\UEmb_{n}(S^1,\RR^2)$ by precomposing $\Gamma_0*\Gamma_0'$ with $\psi_{f_j}$ when $j\preceq i$, this yields
	$$F_0(t,f)\coloneqq\left\{\gamma_{t,f_j}\coloneqq\begin{cases}
		\gamma_{f_i}\circ(\phi_{\varphi^{-1}_{f_i}(r_{f_i})}(2t))\circ\psi_{f_j}&0\leq t\leq1/2\\
		h_{2-2t,{f_i},\varphi^{-1}_{f_i}(r_{f_i})}\circ\psi_{f_j}&1/2\leq t\leq1
	\end{cases} \ \Bigg\vert \ d_i=0 \ \& \ j\preceq i\right\}.$$ 
	
	Now we do the same procedure to $F_0(1,f)\setminus F_0(1,f^{(0)})$ and leave $F_0\left(1,f^{(0)}\right)$ stationary to obtain $F_1$. The first part of $F_1$, that is $\Gamma_1$, contracts $F_0(1,f)\setminus F_0(1,f^{(0)})$ so that $\Gamma_1'$ lies within $\bigcup\limits_{d_j=1, \ j\preceq i} D(h_{0,{f_i},\varphi^{-1}_{f_i}(r_{f_i})}\circ\psi_{f_j})$. We continue this procedure for a total of $d+1$ times and write $[F_0*F_1*\cdots* F_d(t,f)]$ for the image of this homotopy in $\UJ_n(\RR^2)$.
	
	\begin{lemma}\label{lem:mainwelldefineddefret}
		$[F_0*F_1*\cdots* F_d(t,f)]$ is well defined and is a deformation retraction of $\UJ_n(\RR^2)$ to $\CUJ_n(\RR^2)$.
	\end{lemma}
	\begin{proof}
		We check $F_0$ since each $F_k$ does the same process at different levels of nesting. For all $i\neq k$ with $d_i=0=d_k$ we have $\bigcup\limits_{j\preceq i}\gamma_{t,f_j}(S^1)$ and $\bigcup\limits_{j\preceq k}\gamma_{t,f_j}(S^1)$ stay disjoint since $\gamma_{f_i}\circ(\phi_{\varphi^{-1}_{f_i}(r_{f_i})}(2t))$ contracts $\{\gamma_{f_j}\mid j\preceq i\}$ so that $\{h_{2-2t,{f_i},\varphi^{-1}_{f_i}(r_{f_i})}\circ\psi_{f_j}\mid j\preceq i \ \& \ 1/2\leq t\leq 1\}$ lies inside $D(f_i)$ and the same holds for $k$. 
		For each fixed $t\in[0,1/2]$ the map $\gamma_{f_i}\circ(\phi_{\varphi^{-1}_{f_i}(r_{f_i})}(2t))$ is injective and for each fixed $t\in[1/2,1]$ the map $h_{2-2t,{f_i},\varphi_{f_i}^{-1}(r_i)}$ is injective, so $\gamma_{t,f_j}(S^1)$ and $\gamma_{t,f_l}(S^1)$ are disjoint when $j,l\preceq i$ and $j\neq l$. 
		Hence, $\gamma_{t,f_j}(S^1)$ and $\gamma_{t,f_l}(S^1)$ stay disjoint for all $t$ and all $j\neq l$, so $F_0* F_1*\cdots* F_d(t,f)$ defines a homotopy $\UJ_n(\RR^2)\times[0,1]\to\UEmb_n(S^1,\RR^2)$.
		
		We send each curve to a map that only depends on the image of the curve, which is invariant under precomposition with homeomorphisms of a circle.
		The image of $F_k$ in each procedural step is invariant with respect to the chosen representative of a curve and passing to the equivalence class $[F_0*F_1*\cdots* F_d(t,f)]$ is well defined and yields a homotopy of $\UJ_n(\RR^2)$. In particular, the final image of each curve is convex.
	\end{proof}

	\begin{lemma}[Strong deformation retraction]
		The deformation retraction $[F_0*F_1*\cdots* F_d(t,f)]$ is stationary on $\RUJ_n(\RR^2)$.
	\end{lemma}
	\begin{proof}
		Suppose $f$ is a round Jordan configuration. For all $d_i=0$ and $j\preceq i$, we have $\gamma_{f_j}(x)=c(D(f_j))+r_{f_j}x$ and $h_{t,f_i,y}(x)=c(D(f_i))+r_{f_i}x=\gamma_{f_i}(x)$ for all $t$ and $y$, so $\varphi_{f_i}^{-1}(r_{f_i})=0$ and $\gamma_{t,f_j}=\gamma_{f_i}\circ\psi_{f_j}=\gamma_{f_j}$ for all $t$.
	\end{proof}
	
	\begin{theorem}\label{thrm:preroundingJordanconfigurationsintheplane}
		There is a strong deformation retraction of the configuration space of $n$ Jordan curves in the plane, $\UJ_n(\RR^2)$, to the space of round Jordan configurations in the plane, $\RUJ_n(\RR^2)$. Analogous statements hold for $\J_n(\RR^2)$ onto $\RJ_n(\RR^2)$.
	\end{theorem}
	\begin{proof}
		Concatenation of $[F_0* F_1*\cdots* F_d(t,f)]$, which is stationary on circles, with the strong deformation retraction in Theorem~\ref{thrm:convdefretractontojrdncurves} yields the first statement of the theorem. 
		This construction of a strong deformation retraction does not depend on a 
		labeling, so the latter statement of the theorem holds as well.
	\end{proof}
	\begin{lemma}\label{lem:nonequidefretract}
		The deformation retraction $[F_0*\cdots* F_d(t,f)]$ is $\Con(\RR^2)$-equivariant.
	\end{lemma}
	\begin{proof}
		Let $f=\{[f_1],\dots,[f_n]\}$ be a Jordan configuration. It suffices to check that the image of each element of $[F_0(t,f)]$ is $\Con(\RR^2)$-equivariant. Let $g\in\Con(\RR^2)$, and $i,j\in\{1,\dots,n\}$ such that $d_i=0$ and $j\preceq i$. 
		
		First we note that the partial ordering of the indices of Jordan configurations is invariant under conformal maps. This is true since the conformal map $g$ does not affect the nesting of curves, so the abstract tree of a Jordan configuration stays the same after $g$, and hence the partial ordering defined by the abstract tree does not change. 
		
		Next we show equivariance of $h_{t,f_i,y}$ and invariance of the scalar $\varphi_{{f_i}}^{-1}(r_i)$. Checking equivariance is again similar to Belegradek and Ghomi's check in \cite[Prop.~4.1]{belegradek2025point} which follows from equivariance of our chosen centers, $g(\gamma_{f_j}(o))=g(c(D(f_j)))=c(g(D(f_j)))=\gamma_{g(f_j)}(o)$. By Lemma~\ref{lem:conformalmapping}, there is a unique map $B^2\to\RR^2$ that is conformal on $\interior(B^2)$ with image $g(\gamma_{f_j}(B^2))=\gamma_{g(f_j)}(B^2)$ and center $g(\gamma_{f_j}(o))=\gamma_{g(f_j)}(o)$, up to rotation of the derivative at $o$, that is $\gamma_{g(f_j)}^{-1}\circ g\circ \gamma_{f_j}$ is an element of $O(2)$. Hence, we have $g(\gamma_{f_j}(tS^1))=\gamma_{g(f_j)}(tS^1)$ for all $t$. Since $\Con(\RR^2)$ is generated by conformal linear maps and translations, we have two cases to check.
		For $t\in(0,1]$, if $g$ is a conformal linear map then
		\begin{align*}
			g(h_{t,f_i,y})&=g\left(c(D(f_i))+\frac{\gamma_{f_i}\left(t(\phi_{y}(1))\right)-c(D(f_i))}{t}\right)\\
			&=g(c(D(f_i)))+\frac{\left(g\circ\gamma_{f_i}\right)\left(t(\phi_{y}(1))\right)-g(c(D(f_i)))}{t}\\
			&=c(g(D(f_i)))+\frac{\left(\gamma_{g(f_i)}\circ (\gamma_{g(f_i)}^{-1}\circ g\circ\gamma_{f_i})\right)\left(t(\phi_{y}(1))\right)-c(g(D(f_i)))}{t}\\
			&=c(D(g(f_i)))+\frac{\gamma_{g(f_i)}\left(t(\phi_{y}(1))\circ(\gamma_{g(f_i)}^{-1}\circ g\circ\gamma_{f_i})\right)-c(D(g(f_i)))}{t}\tag{*}\label{eq:rotation}\\
			&=h_{t,g(f_i),y}\circ(\gamma_{g(f_i)}^{-1}\circ g\circ\gamma_{f_i}),
		\end{align*} so $g(h_{t,f_i,y})(S^1)=h_{t,g(f_i),y}(S^1)$ for all $y$; if $g$ is translation by $(a,b)$ then
		\begin{align*}
			g(h_{t,f_i,y})	&=(a,b)+c(D(f_i))+\frac{\gamma_{f_i}\left(t(\phi_{y}(1))\right)-c(D(f_i))}{t}\\
			&=(a,b)+c(D(f_i))+\frac{(a,b)+\gamma_{f_i}\left(t(\phi_{y}(1))\right)-(a,b)-c(D(f_i))}{t}\\
			&=c(D(g(f_i)))+\frac{\gamma_{g(f_i)}\left(t(\phi_{y}(1))\right)-c(D(g(f_i)))}{t}\\
			&=h_{t,g(f_i),y},
		\end{align*} so $g(h_{t,f_i,y})(S^1)=h_{t,g(f_i),y}(S^1)$ for all $y$, and by continuity we have $\Con(\RR^2)$-equivariance of $h_{t,f_i,y}(S^1)$ for all $t$. 
		Note that line \ref{eq:rotation} is due to the fact that rotation and dilation commute when centered about the origin, $o$.
		Since conformal maps of the plane preserve ratios of distances, for any $g\in\Con(\RR^2)$ we have $$\frac{\varphi_{f_i}(y)}{r_{f_i}}
		=\frac{\varphi_{g(f_i)}(y)}{r_{g(f_i)}}.$$
		Since $\varphi^{-1}_{f_i}(r_{f_i})$ is the unique $y$ value such that $\varphi_{f_i}(y)=r_{f_i}$, it is also the unique $y$ value such that $\varphi_{g(f_i)}(y)=r_{g(f_i)}$ and hence, invariant under conformal transformations, that is $\varphi^{-1}_{f_i}(r_{f_i})=\varphi^{-1}_{g(f_i)}(r_{g(f_i)})$.
		
		Now that we see the scalar $\varphi_{f_i}^{-1}(r_i)$ is $\Con(\RR^2)$-invariant we move to the equivariance of $F_0$. Suppose $g$ is a conformal linear map, 
		if $t\in[0,1/2]$ then
		\begin{align*}
			g(\gamma_{t,f_j})&=g(\gamma_{f_i}\circ\phi_{\varphi^{-1}_{f_i}(r_{f_i})}(2t)\circ\psi_{f_j})\\
			&=\gamma_{g(f_i)}\circ(\gamma_{g(f_i)}^{-1}\circ g\circ \gamma_{f_i})\circ\phi_{\varphi^{-1}_{f_i}(r_{f_i})}(2t)\circ\psi_{f_j}\\
			&=\gamma_{g(f_i)}\circ\phi_{\varphi^{-1}_{f_i}(r_{f_i})}(2t)\circ(\gamma_{g(f_i)}^{-1}\circ g\circ \gamma_{f_i})\circ\psi_{f_j}\tag{**}\label{rotdil}\\
			&=\gamma_{g(f_i)}\circ\phi_{\varphi^{-1}_{f_i}(r_{f_i})}(2t)\circ(\gamma_{g(f_i)}^{-1}\circ g\circ \gamma_{f_i})\circ(\gamma_{f_i}^{-1}\circ\gamma_{f_j})\\
			&=\gamma_{g(f_i)}\circ\phi_{\varphi^{-1}_{f_i}(r_{f_i})}(2t)\circ(\gamma_{g(f_i)}^{-1}\circ g\circ\gamma_{f_j})\\
			&=\gamma_{g(f_i)}\circ\phi_{\varphi^{-1}_{f_i}(r_{f_i})}(2t)\circ\gamma_{g(f_i)}^{-1}\circ (\gamma_{g(f_j)}\circ(\gamma_{g(f_j)}^{-1}\circ g\circ \gamma_{f_j}))\\
			&=\gamma_{g(f_i)}\circ\phi_{\varphi^{-1}_{f_i}(r_{f_i})}(2t)\circ\psi_{g(f_j)}\circ(\gamma_{g(f_j)}^{-1}\circ g\circ \gamma_{f_j})\\
			&=\gamma_{t,g(f_j)}\circ(\gamma_{g(f_j)}^{-1}\circ g\circ \gamma_{f_j}),
		\end{align*}
		where line~\ref{rotdil} again happens because rotation and dilation commute when both are centered at $o$, so we have $g(\gamma_{t,f_j}(S^1))=\gamma_{t,g(f_j)}(S^1)$;
		if $t\in[1/2,1)$ then
		\begin{align*}
			g(\gamma_{t,f_j})=&g(h_{2-2t,f_i,\varphi_{f_i}^{-1}(r_{f_i})}\circ\psi_{f_j})\\
			=&h_{2-2t,g(f_i),\varphi_{g(f_i)}^{-1}(r_{g(f_i)})}\circ(\gamma_{g(f_i)}^{-1}\circ g\circ\gamma_{f_i})\circ\psi_{f_j}\\
			=&h_{2-2t,g(f_i),\varphi_{g(f_i)}^{-1}(r_{g(f_i)})}\circ\psi_{g(f_j)}\circ(\gamma_{g(f_j)}^{-1}\circ g\circ\gamma_{f_j})\\
			=&\gamma_{t,g(f_j)}\circ(\gamma_{g(f_j)}^{-1}\circ g\circ\gamma_{f_j}),
		\end{align*}
		so we have $g(\gamma_{t,g(f_j)}(S^1))=\gamma_{t,g(f_j)}(S^1)$. 
		Suppose $g$ is translation by $(a,b)$, if $t\in[0,1/2]$ then
		\begin{align*}
			g(\gamma_{t,f_j})&=(a,b)+\gamma_{f_i}\circ\phi_{\varphi^{-1}_{f_i}(r_{f_i})}(2t)\circ\psi_{f_j}\\
			&=\gamma_{g(f_i)}\circ\phi_{\varphi^{-1}_{g(f_i)}(r_{g(f_i)})}(2t)\circ\psi_{g(f_j)}\\
			&=\gamma_{t,g(f_j)};
		\end{align*}
		if $t\in[1/2,1)$ then
		\begin{align*}
			g(\gamma_{t,f_j})&=(a,b)+h_{2-2t,f_i,\varphi_{f_i}^{-1}(r_{f_i})}\circ\psi_{f_j}\\
			&=h_{2-2t,g(f_i),\varphi_{g(f_i)}^{-1}(r_{g(f_i)})}\circ\psi_{g(f_j)}\\
			&=\gamma_{t,g(f_j)}.
		\end{align*}
		By continuity in $t$, we have equivariance.
	\end{proof}
	
	\begin{theoremB}[Part 1]
		There is a $\Con(\RR^2)$-equivariant strong deformation retraction of  $\UJ_n(\RR^2)$ onto $\RUJ_n(\RR^2)$. Analogous statements hold for $\J_n(\RR^2)$ onto $\RJ_n(\RR^2)$.
	\end{theoremB}
	\begin{proof}
		Concatenation of $[F_0*\cdots* F_d(t,f)]$ with the $\Con(\RR^2)$-equivariant strong deformation retraction in Theorem~\ref{thrm:convdefretractontojrdncurves} yields the first statement of the theorem. 
		This construction of a $\Con(\RR^2)$-equivariant strong deformation retraction does not depend on a 
		labeling, so the latter statement of the theorem holds as well.
	\end{proof}
	
	\begin{remark}
		Theorem~\ref{thrm:convdefretractontojrdncurves} is stronger than what is needed in the proof of Theorem~\ref{thrm:preroundingJordanconfigurationsintheplane} and Theorem~\ref{thrm:roundingcurvesinsurfaces}~Part~1 because the deformation retraction $[F_0*\cdots* F_d]$ transforms Jordan configurations into ellipses.
	\end{remark}
	
	\begin{corollary}\label{cor:he}
		$\UJ_n(\RR^2)$ and $\UJ_n(\interior(B^2))$ are homotopy equivalent to $\UConf_n(S^1,\RR^2)$. Similarly,
		$\J_n(\RR^2)$ and $\J_n(\interior(B^2))$ are homotopy equivalent to $\Conf_n(S^1,\RR^2)$.
	\end{corollary}
	\begin{proof}
		This follows immediately from Theorems~\ref{thm:embeddingconfspaces} ~and~\ref{thrm:preroundingJordanconfigurationsintheplane}.
	\end{proof}
	
	Restricting this homotopy equivalence to the connected components allows us to define the following.
	
	\begin{definition}\label{dfn:jt}
		Let $X$ be homeomorphic to $\RR^2$. For a tree $T$, define the \emph{space of $T$-Jordan configurations}, $\UJ_T(X)$, to be the connected component of $\UJ_n(X)$ homotopy equivalent via the homotopy equivalence from Corollary~\ref{cor:he}, to $\UConf_T(S^1,\RR^2)$. For a labeled tree $T$, define the \emph{space of labeled $T$-Jordan configurations}, $\J_T(X)$, to be the connected component of $\J_n(X)$ homotopy equivalent via the homotopy equivalence from Corollary~\ref{cor:he}, to $\Conf_T(S^1,\RR^2)$.
	\end{definition}
	
	
	\subsection{Jordan configurations in non-positively curved surfaces}\label{sec:g>0}
	
	
	The following is a generalization of Belegradek and Ghomi's second application of Theorem~\ref{thrm:extequcenter} \cite[Thm. 4.2]{belegradek2025point}. Here we take their construction of an $\Iso(X)$-equivariant deformation retraction of $\mathcal{D}(X)$ onto the space of round Jordan domains in $X$, and update it to $\Iso(X)$-equivariantly deformation retract $\UJ_n(X)$ onto $\RUJ_n(X)$.
	
	\begin{theoremB}[Part 2]
		Let $X$ be a geodesically complete connected surface of constant non-positive curvature. Then $\UJ_n(X)$ admits an $\Iso(X)$-equivariant strong deformation retraction onto $\RUJ_n(X)$. Analogous statements hold for $\J_n(X)$ onto $\RJ_n(X)$.
	\end{theoremB}
	\begin{definition}[Exponential map]\label{def:exp}
		Let $M$ be a complete connected Riemannian surface. Given a vector $v\in T_pM$, there is a unique geodesic starting at $p$, going in the direction of $v$, with speed $|v|$. The exponential map at $p$, written $\exp_p:T_pM\to M$, sends a vector $v\in T_pM$ to the terminal point of the unique geodesic associated to $v$ after traveling one unit of time.
	\end{definition}
	\begin{proof}[Proof of Theorem B (Part 2)]
		First, by Theorem~\ref{killhopfthrm} we may work with $\widetilde{X}$, the complete universal Riemannian cover of $X$ with constant curvature, where $\Iso(X)$ lifts to $G\leq\Iso(\widetilde{X})$. By Corollary~\ref{cor:existcenter} there is an $\Iso(\widetilde{X})$-equivariant center $c$ on $\mathcal{D}(\widetilde{X})$ which coincides with centers of circles. For $f\in \UJ_n(\widetilde{X})$ there is still the strict partial ordering defined by the abstract tree of $f$. 
		
		For $d_i=0$, we use the inverse exponential map $\exp^{-1}_{c(D(f_i))}$ to lift the curves $[f_j]$ with $j\preceq i$ to the tangent plane $T_{c(D(f_i))}(\widetilde{X})$. 
		Using $F_0$ from Subsection~\ref{sec:proof} and Theorem~\ref{thrm:convdefretractontojrdncurves} on $\exp^{-1}_{c(D(f_i))}(\{[f_j]:j\preceq i\})$ we get a deformation of curves contained in $\exp^{-1}_{c(D(f_i))}(D(f_i))$ that is constant if $D(f_i)$ is already round. Applying $\exp_{c(D(f_i))}$ yields a deformation of $\{[f_j]:j\preceq i\}$ contained in $D(f_i)$ that depends continuously on $f$, makes $[f_i]$ round, and is constant if $[f_i]$ is already round. Then we repeat this procedure to the resulted deformation of $[f_j]$ such that $j\preceq k$ for $d_k=1$, leaving the rounded $[f_i]$ stationary. We repeat this procedure for a total of $d+1$ times obtaining a strong deformation retraction of $\UJ_n(\widetilde{X})$ onto $\RUJ_n(\widetilde{X})$.
		
		Next, we see that this strong deformation retraction is equivariant by checking each step $i$. Let $\rho$ be an isometry of $\widetilde{X}$. Then $\rho$ ascends to the pushforward $d\rho_{c(D(f_i))}:T_{c(D(f_i))}\widetilde{X}\to T_{\rho(c(D(f_i)))}\widetilde{X}=T_{c(D(\rho(f_i)))}\widetilde{X}$ which is again an isometry. However, the deformation on the tangent planes from Theorem~\ref{thrm:roundingcurvesinsurfaces}~Part~1 is $\Iso(\RR^2)$-equivariant, so the deformation of $f$ corresponds through $\rho$ to the deformation of $\rho(f)$.
		
		Since this $\Iso(\widetilde{X})$-equivariant strong deformation retraction is contained in $\bigcup_{d_i=0}D(f_i)$ it descends to an $\Iso(X)$-equivariant strong deformation retraction of $\UJ_n(X)$ onto $\RUJ_n(X)$.
	\end{proof}
	\begin{remark}
		Belegradek and Ghomi's deformation retraction sends Jordan domains to the largest round Jordan domain contained in the original domain centered at the domain's center. Our deformation is only constant on round Jordan configurations and does not guarantee that an un-nested curve is sent to the largest circle centered at the curve's center contained in the curve's original bounded domain. However, these deformation retractions could be amended to send Jordan configurations to an analogue of the largest round Jordan configuration contained within the domains the curves bound.
	\end{remark}
	
	\begin{definition}
		The \emph{injectivity radius} of a point $x\in X$ is the radius of the largest ball centered at $x$ such that the exponential map is a diffeomorphism.
	\end{definition}
	
	\begin{corollary}\label{cor:hebetweenconfpointandnonnestedcomponents}
		There is a homotopy equivalence between $\UConf_n(*,X)$ and the connected component(s) of $\UJ_n(X)$ with no nested curves.
	\end{corollary}
	\begin{proof}
		Let $\chi$ be the connected component of $\RUJ_n(X)$ with no nested curves. For each $f=\{[f_1],\dots,[f_n]\}\in \chi$, define $\Phi:\chi\to \UConf_T(*,X)$ to send $f$ to $\{c(D(f_1)),\dots,c(D(f_n))\}$, then $\Phi$ is well defined since no two circles are nested inside each other and therefore do not have the same center. For $x=\{x_1,\dots,x_n\}\in\UConf_n(*,X)$, let $r$ be the minimum of the pairwise distances and their injectivity radii, and define $\Psi:\UConf_n(*,X)\to \UJ_n(X)$ to send $x$ to the round Jordan configuration centered at these points with radius $r/3$. Then $\Phi\circ\Psi$ is the identity and the homotopy from $\Psi\circ\Phi$ to the identity on $\RUJ_T(X)$ is given by using the exponential map to geodesically shrink round curves that have a larger radius than what they started with and then geodesically growing the round curves that have a smaller radius than what they started with.
	\end{proof}
	
	In particular, there is a unique connected component of $\UJ_n(X)$ where no curves are nested inside of each other.

		\begin{theoremA}[Indexing connected components]
			The abstract trees of Jordan configurations induce a one-to-one correspondence between the collection of finite rooted trees with $n$ non-root vertices and the connected components of $\UJ_n(X)$ such that every tree is the abstract tree of each Jordan configuration in its associated connected component of $\UJ_n(X)$. 
			The analogous result holds for the labeled case.
		\end{theoremA}
		\begin{proof}
			By Lemma~\ref{lem:pathconnembsgivethesameabstracttree} each Jordan configuration in a connected component of $\UJ_n(X)$ has the same abstract tree. We need to prove that each tree is associated to a unique connected component of $\UJ_n(X)$.
			
			Let $T$ be an abstract tree with $n$ non-root vertices and $D$ an open ball in $X$. Then there exists a unique connected component $\chi_T$ of $\UJ_n(D)$ such that $T$ is the abstract tree for all Jordan configurations in $\chi_T$ by Proposition~\ref{prop:components} and Corollary~\ref{cor:he}. 
			We have $\chi_T$ is a subspace of a connected component of $\UJ_n(X)$. By Lemma~\ref{lem:pathconnembsgivethesameabstracttree} any Jordan configuration in the connected component of $\UJ_n(X)$ containing $\chi_T$ has abstract tree $T$.
			
			Now we will show for any abstract tree $T$ there is only one connected component of $\UJ_{n}(X)$ where all Jordan configurations in the connected component have abstract tree $T$; that is, two Jordan configurations with the same abstract tree are connected by a path. Suppose $f$ and $f'$ are elements of $\UJ_{n}(X)$ with abstract tree $T$. We will induct on $n$ to show that $f$ and $f'$ are connected by a path in $\UJ_n(X)$.
			The base case is when $n=0$, then $f$ and $f'$ are both the empty set and they have the abstract tree that does not contain any non-root vertices; we also have that $f$ and $f'$ are connected by a path trivially, they are connected by the constant path. 
			
			Suppose for some $k$ and for all $n$ such that $0\leq n\leq k$, two Jordan configurations in $\UJ_n(X)$ with equal abstract trees are connected by a path. Let $f=\{[f_1],\dots,[f_{k+1}]\}$ and $f'=\{[f_1'],\dots,[f_{k+1}']\}$ be two Jordan configurations with the same abstract tree $T$. By Theorem~\ref{thrm:roundingcurvesinsurfaces} we can assume $f$ and $f'$ are round Jordan configurations. Since there are finitely many vertices of $T$, there exists a vertex, $v$, with children, $\{v_1,\dots,v_m\}$, such that no $v_i$ has a child of its own for $1\leq i\leq m$. Let $T(v)$ be the subtree of $T$ rooted at $v$ with children $\{v_1,\dots,v_m\}$. The vertex $v$ corresponds to two curves $[f_{r_v}]$ and $[f_{s_v}']$ in $f$ and $f'$, the children of $v$ correspond to the Jordan configurations $\{[f_{j}]\mid j\prec r_v\}$ and $\{[f_j']\mid j\prec s_v\}$ nested inside $[f_{r_v}]$ and $[f_{s_v}']$, respectively, and no $v_i$ having a child of its own for $1\leq i\leq m$ means that there are no nested curves in $\{[f_{j}]\mid j\prec r_v\}$ or $\{[f_j']\mid j\prec s_v\}$. Furthermore, removing $\{v_1,\dots,v_m\}$ from $T$ corresponds to $f\setminus\{[f_{j}]\mid j\prec r_v\}$ and $f'\setminus\{[f_j']\mid j\prec s_v\}$. 
			
			If $v$ is the root, by Corollary~\ref{cor:hebetweenconfpointandnonnestedcomponents}, we have that $f$ and $f'$ are elements of a space homotopy equivalent to the path connected space $\UConf_{k+1}(*,X)$, and hence are connected by a path. 
			If $v$ is not the root, then there exists a path 
			$$p_t:=\{[f_{1,t}],\dots,[f_{k+1-m,t}]\}$$
			from $f\setminus\{[f_{j}]\mid j\prec r_v\}$ to $f'\setminus\{[f_j']\mid j\prec s_v\}$ by the induction hypothesis. Let $p_t$ be indexed to agree with $f$, that is $[f_{j,0}]=[f_j]$. Then $p_t\cup\{[\gamma_{f_{r_v,t}}\circ\gamma_{f_{r_v}}^{-1}\circ f_j]\mid j\prec r_v\}$ is a path from $f$ to some Jordan configuration that agrees with $f'\setminus\{[f_j']\mid j\prec s_v\}$. 
			Since $\UJ_{T(v)}(D(f_{s_v}'))$ is homotopy equivalent to the path connected space $\UConf_{T(v)}(S^1,\RR^2)$ by Corollary~\ref{cor:he}, there is a path from $\{[\gamma_{f_{r_v,1}}\circ\gamma_{f_{r_v}}^{-1}\circ f_j]\mid j\prec r_v\}$ to $\{[f_j']\mid j\prec s_v\}$.
			So there is a path from $f$ to $f'$ that is $p_t$ on $f\setminus\{[f_{j}]\mid j\prec r_v\}$ to $f'\setminus\{[f_j']\mid j\prec s_v\}$ and some path from $\{[f_{j}]\mid j\prec r_v\}$ to $\{[f_j']\mid j\prec s_v\}$ which is contained in $D(f_{r_v,t})$ for each $t$, and hence is a path in $\UJ_n(X)$.
		\end{proof}
		
	\begin{definition}\label{dfn:tj}
		For a tree $T$, define the \emph{space of $T$-Jordan configurations}, $\UJ_T(X)$, to be the connected component of $\UJ_n(X)$ where each Jordan configuration has abstract tree $T$. For a labeled tree $T$, define the \emph{space of labeled $T$-Jordan configurations}, $\J_T(X)$, to be the connected component of $\J_n(X)$ where each Jordan configuration has abstract labeled tree $T$.
	\end{definition}
	
	
		\section{Homotopy groups}\label{sec:homotopygroups}
		
		In this section we prove that these spaces of Jordan configurations in surfaces are $K(\pi,1)$s, that is, homotopy equivalent to a CW-complex and aspherical; as well as prove that their fundamental group decomposes as a semi-direct product of a surface braid group and a product of braided automorphism groups of trees.
		
		For a vertex $v_i$ in a tree $T$ write $T(v_i)$ for the maximal subtree of $T$ rooted at $v_i$. If $T$ is planar, then $T(v_i)$ inherits the planar ordering. Write $\Aut(T)$ for the automorphism group of a planar tree $T$.
		
		\begin{proposition}[Curry, Gelnett, Zaremsky, \cite{Curry2026}]\label{prop:decomp_Aut(T)}
			Let $T$ be a planar tree. With $\{v_1,\dots,v_m\}$ the set of children of the root of $T$, let $\Pi$ be the partition of $\{1,\dots,m\}$ such that $i$ and $j$ share a block of $\Pi$ if and only if $T(v_i)\cong T(v_j)$. The automorphism group $\Aut(T)$ decomposes as a semidirect product
			\[
			\Aut(T) \cong (\Aut(T(v_1))\times\cdots\times\Aut(T(v_m)))\rtimes \Sigma_m^\Pi \text{,}
			\]
			where $\Sigma_m^\Pi$ is the subgroup of the symmetric group of $m$ elements consisting of all permutations that setwise stabilize each element of $\Pi$. The action of $\Sigma_m^\Pi$ on the direct product is by permuting entries.
		\end{proposition}
		
		Recall for a connected topological 2-manifold $X$ different from $S^2$ and $\RR P^2$, the configuration space of points on $X$ is aspherical~\cite{fadell1962configuration} with a surface braid group on $m$ strands defined to be $B_m(X):=\pi_1(\UConf_m(*,X))$. When $X=\RR^2$ we recover the usual braid groups.
	 	There is a natural projection, $\pi:B_m(X)\to\Sigma_m$ given by forgetting over/under crossings and recording where strands start and end. Since $\Sigma_m^{\Pi}$ defined in Proposition~\ref{prop:decomp_Aut(T)} is a subgroup of $\Sigma_m$, we may define $B_m^{\Pi}(X)$ to be the preimage of $\Sigma_m^{\Pi}$ through $\pi$. Then $B_m^{\Pi}(X)$ is the subgroup of $B_m(X)$ where strands start and end in the same block of $\Pi$, although strands may cross with other strands that start and end at other blocks of $\Pi$. Finally, recall that the pure surface braid group on $m$ strands, $PB_m(X)$ is the kernel of $\pi$, which is the subgroup of $B_m(X)$ where each strand starts and ends at the same point.
		
		\begin{definition}[Braided $\Aut(T)$]\label{def:BPBAut}
			Let $T$ be a planar tree. With $\{v_1,\dots,v_m\}$ and $\Pi$ as in Proposition~\ref{prop:decomp_Aut(T)}, we define the \emph{braided automorphism group of $T$} to be
			\[
			\BAut(T) \coloneq (\BAut(T(v_1))\times\cdots\times \BAut(T(v_m))) \rtimes B_m^\Pi(\RR^2) \text{,}
			\]
			where the action of $B_m^{\Pi}(\RR^2)$ on the direct product is given by mapping to $\Sigma_m^\Pi$ and then permuting the entries.
			
			The \emph{pure braided automorphism group} $\PBAut(T)$ is the kernel of the map $\pi\colon \BAut(T)\to \Aut(T)$. Since $PB_m(\RR^2)$ acts trivially when permuting entries, $\PBAut(T)$ decomposes as a direct product
			\[
			\PBAut(T) = (\PBAut(T(v_1))\times\cdots\times \PBAut(T(v_m))) \times PB_m(\RR^2) \text{,}
			\]
			and so $\PBAut(T)$ decomposes as
			\[
			\PBAut(T) \cong PB_{m_1}(\RR^2)\times\cdots\times PB_{m_n}(\RR^2)\times PB_m(\RR^2) \text{,}
			\]
			where $m_i$ is the number of children of the vertex labeled $i$, under some arbitrarily chosen labeling of $T$.
		\end{definition}
		
		Note when $T$ is a labeled planar tree $\Aut(T)$ is trivial since no two vertices are labeled the same; for the same reason we have $\Pi$, as defined in Proposition~\ref{prop:decomp_Aut(T)} and Definition~\ref{def:BPBAut}, is the discrete partition. We may construct $\BAut(T)$ and $\PBAut(T)$ for $T$ a labeled planar tree in the same way as an unlabeled planar tree. For $T$ a labeled planar tree, let $U(T)$ be the underlying unlabeled planar tree.
		
		\begin{lemma}\label{lem:labeledBAut}
			For $T$ a labeled planar tree, we have $\BAut(T)\cong\PBAut(T)\cong\PBAut(U(T))$.
		\end{lemma}
		\begin{proof}
			Let $\{v_1,\dots,v_m\}$ be the set of children of the root of $T$. $\PBAut(U(T))$ and $\PBAut(T)$ decompose as $PB_{m_1}(\RR^2)\times\cdots\times PB_{m_n}(\RR^2)\times PB_m(\RR^2)$ where $m_i$ is the number of children of the vertex labeled $i$. Since $\Pi$ is the discrete partition of $\{1,\dots,m\}$, the group $B^{\Pi}_m(\RR^2)$ is the subgroup of $B_m(\RR^2)$ where each strand starts and ends at the same position, i.e. $PB_m(\RR^2)$. By induction $\BAut(T)$ decomposes as $PB_{m_1}(\RR^2)\times\cdots\times PB_{m_n}(\RR^2)\times PB_m(\RR^2)$.
		\end{proof}
		
		The primary result of \cite{Curry2026} is that $\PBAut(T)$ and $\BAut(T)$ are in fact the fundamental groups of labeled and unlabeled configuration spaces of circles in the plane, respectively.
		
		\begin{theorem}[\cite{Curry2026}]\label{thrm:confcircunlabclassifying}
			Let $T$ be a planar tree with $n$ non-root vertices. Then $\UConf_T(S^1,\RR^2)$ is aspherical, and its fundamental group is isomorphic to $\BAut(T)$, hence it is a $K(\BAut(T),1)$.
		\end{theorem}
		\begin{theorem}[\cite{Curry2026}]\label{thrm:confcirclabclassifying}
			Let $T$ be a labeled planar tree with $n$ non-root vertices. Then $\Conf_T(S^1,\RR^2)$ is aspherical, and its fundamental group is isomorphic to $\PBAut(T)$. By Lemma~\ref{lem:labeledBAut} we have $\Conf_T(S^1,\RR^2)$ is a $K(\BAut(T),1)$.
		\end{theorem}
		
		These are proven by the use of a long exact sequence of homotopy groups pertaining to the following fiber bundle: $$F\to\UConf_T(S^1,\RR^2)\overset{\pi}{\to}\UConf_{\Lambda}(S^1,\RR^2)$$ where \begin{enumerate}
			\item $T$ is a tree,
			\item $\Lambda$ is the subtree of $T$ containing only the root and the children of the root,
			\item $\pi$ deletes the nested circles of a configuration of circles in the plane, and
			\item $F$ is the subspace of $\UConf_T(S^1,\RR^2)$ where the collection of un-nested circles is the fixed configuration corresponding to $\Lambda$.
		\end{enumerate}
		
		The map on trees induced by $\pi$ is given by deleting vertices that are not the root or children of the root, so $T$ is sent to $\Lambda$. The base case is when $T$ is already a tree with all non-root vertices children of the root, then $\UConf_T(S^1,\RR^2)$ consists of configurations of un-nested circles and Proposition~\ref{prop:hetoconfofpoints} tells us the homotopy groups are known. The base space $\UConf_{\Lambda}(S^1,\RR^2)$ contributes $B^{\Pi}_m(\RR^2)\leq B_m(\RR^2)$ to $\BAut(T)$, where circles start and end at circles with the same nesting order which is observed by $\Pi$. The fiber $F$ contributes to the recursive decomposition of $\BAut(T)$. Figure~\ref{fig:fiber_bundle_braid} is a colorful illustration of this fiber bundle where the un-nested circles are black, nested circles are red and blue, and the fiber bundle decomposes a braid in $\UConf_T(S^1,\RR^2)$ into a braid of un-nested circles and separately a braid of nested circles. 
		Theorem~\ref{thrm:confcirclabclassifying} follows from Theorem~\ref{thrm:confcircunlabclassifying} and the fact that $\Conf_T(S^1,\RR^2)\to\UConf_T(S^1,\RR^2)$ is a regular covering with deck group $\Aut(T)$.
		
		\begin{figure}[hb]
			\centering
			\begin{tikzpicture}[line width=1pt,yscale=.6]
				
				\draw[white,line width=7pt] (-0.27,0) to[out=-90, in=90] (0.77,-7);
				\draw[blue] (-0.27,0) to[out=-90, in=90] (0.77,-7);
				\draw[white,line width=7pt] (0.27,0) to[out=-90, in=90] (2.2,-7);
				\draw[blue] (0.27,0) to[out=-90, in=90] (2.2,-7);
				
				\draw[white,line width=3pt] (0.77,0) to[out=-90, in=90] (-0.27,-7);
				\draw[blue] (0.77,0) to[out=-90, in=90] (-0.27,-7);
				\draw[white,line width=3pt] (2.2,0) to[out=-90, in=90] (0.27,-7);
				\draw[blue] (2.2,0) to[out=-90, in=90] (0.27,-7);
				
				\draw[white,line width=5pt] (-0.95,0) -- (-0.95,-7)   (2.95,0) -- (2.95,-7);
				\draw (-0.95,0) -- (-0.95,-7)   (2.95,0) -- (2.95,-7);
				
				\draw[blue] (0,0) ellipse (7.7pt and 3.5pt);
				\draw[blue] (1.485,0) ellipse (20.3pt and 4pt);
				\draw (1,0) ellipse (55.5pt and 7.5pt);
				
				\draw[blue] (0,-7) ellipse (7.7pt and 4pt);
				\draw[blue] (1.488,-7) ellipse (20.4pt and 4pt);
				\draw (.999,-7) ellipse (55.5pt and 8pt);
				
				\draw[white,line width=3pt] (7.27,0) to[out=-90, in=90] (6.23,-7);
				\draw[red] (7.27,0) to[out=-90, in=90] (6.23,-7);
				\draw[white,line width=3pt] (8.7,0) to[out=-90, in=90] (6.77,-7);
				\draw[red] (8.7,0) to[out=-90, in=90] (6.77,-7);
				
				\draw[white,line width=3pt] (6.23,0) to[out=-90, in=90] (7.27,-7);
				\draw[red] (6.23,0) to[out=-90, in=90] (7.27,-7);
				\draw[white,line width=3pt] (6.77,0) to[out=-90, in=90] (8.7,-7);
				\draw[red] (6.77,0) to[out=-90, in=90] (8.7,-7);
				
				\draw[white,line width=5pt] (5.55,0) -- (5.55,-7)   (9.45,0) -- (9.45,-7);
				\draw (5.55,0) -- (5.55,-7)   (9.45,0) -- (9.45,-7);
				
				\draw[red] (6.5,0) ellipse (7.65pt and 4pt);
				\draw[red] (7.985,0) ellipse (20.38pt and 4pt);
				\draw (7.5,0) ellipse (55.5pt and 8pt);
				
				\draw[red] (6.5,-7) ellipse (7.65pt and 4pt);
				\draw[red] (7.985,-7) ellipse (20.38pt and 4pt);
				\draw (7.5,-7) ellipse (55.49pt and 8pt);
				
				\begin{scope}[yscale=10/6,yshift=-.1\textheight,xshift=.7\linewidth,scale=1.2]
					\draw[blue] (0,0) -- (0.5,-0.5) -- (1,0);
					\draw[red] (1.5,0) -- (2,-0.5) -- (2.5,0);
					
					\draw[blue,->] (0,.5) to [bend left] (1,.5);
					\draw[blue, bend at end,<-] (0,.3) to [bend right] (1,.3);
					
					\draw[red,<-] (1.5,.5) to [bend left] (2.5,.5);
					\draw[red, bend at end,->] (1.5,.3) to [bend right] (2.5,.3);
					
					\filldraw[blue] (0,0) circle (1.5pt);
					\filldraw (0.5,-0.5) circle (2.5pt);
					\filldraw[blue] (1,0) circle (1.5pt);
					\filldraw[red] (1.5,0) circle (1.5pt);
					\filldraw (2,-0.5) circle (2.5pt);
					\filldraw[red] (2.5,0) circle (1.5pt);
				\end{scope}
				
			\end{tikzpicture}\hspace{1.5cm}
			
			\
			
			\begin{tikzpicture}[line width=1pt,yscale=.9]
				
				\draw[white,line width=7pt] (-0.27,0) to[out=-60, in=120] (7.27,-7);
				\draw[blue] (-0.27,0) to[out=-60, in=120] (7.27,-7);
				\draw[white,line width=7pt] (0.27,0) to[out=-60, in=120] (8.7,-7);
				\draw[blue] (0.27,0) to[out=-60, in=120] (8.7,-7);
				
				\draw[white,line width=3pt] (0.77,0) to[out=-60, in=120] (6.23,-7);
				\draw[blue] (0.77,0) to[out=-60, in=120] (6.23,-7);
				\draw[white,line width=3pt] (2.2,0) to[out=-60, in=120] (6.77,-7);
				\draw[blue] (2.2,0) to[out=-60, in=120] (6.77,-7);
				
				\draw[white,line width=5pt] (-0.95,0) to[out=-60, in=120] (5.55,-7)   (2.95,0) to[out=-60, in=120] (9.45,-7);
				\draw (-0.95,0) to[out=-60, in=120] (5.55,-7)   (2.95,0) to[out=-60, in=120] (9.45,-7);
				
				\draw[blue] (0,0) ellipse (7.7pt and 3.5pt);
				\draw[blue] (1.485,0) ellipse (20.3pt and 4pt);
				\draw (1,0) ellipse (55.5pt and 7.5pt);
				
				\draw[white,line width=3pt] (7.27,0) to[out=-120, in=60] (-.27,-7);
				\draw[red] (7.27,0) to[out=-120, in=60] (-.27,-7);
				\draw[white,line width=3pt] (8.7,0) to[out=-120, in=60] (.27,-7);
				\draw[red] (8.7,0) to[out=-120, in=60] (.27,-7);
				
				\draw[white,line width=3pt] (6.23,0) to[out=-120, in=60] (.77,-7);
				\draw[red] (6.23,0) to[out=-120, in=60] (.77,-7);
				\draw[white,line width=3pt] (6.77,0) to[out=-120, in=60] (2.2,-7);
				\draw[red] (6.77,0) to[out=-120, in=60] (2.2,-7);
				
				\draw[white,line width=5pt] (5.55,0) to[out=-120, in=60] (-.95,-7)   (9.45,0) to[out=-120, in=60] (2.95,-7);
				\draw (5.55,0) to[out=-120, in=60] (-.95,-7)   (9.45,0) to[out=-120, in=60] (2.95,-7);
				
				\draw[red] (6.5,0) ellipse (7.65pt and 4pt);
				\draw[red] (7.985,0) ellipse (20.38pt and 4pt);
				\draw (7.5,0) ellipse (55.5pt and 8pt);
				
				\draw[blue] (6.5,-7) ellipse (7.65pt and 4pt);
				\draw[blue] (7.985,-7) ellipse (20.38pt and 4pt);
				\draw (7.5,-7) ellipse (55.49pt and 8pt);
				
				\draw[red] (0,-7) ellipse (7.7pt and 4pt);
				\draw[red] (1.488,-7) ellipse (20.4pt and 4pt);
				\draw (.999,-7) ellipse (55.5pt and 8pt);
				
				\begin{scope}[yscale=10/9,yshift=-.1\textheight,xshift=.7\linewidth,scale=1.2]
					\draw[blue] (0,0) -- (0.5,-0.5) -- (1,0);
					\draw[red] (1.5,0) -- (2,-0.5) -- (2.5,0);
					\draw (0.5,-0.5) -- (1.25,-1) -- (2,-.5);
					
					\draw[->] (0.5,-1.5) to [bend left] (2,-1.5);
					\draw[bend at end,<-] (0.5,-1.7) to [bend right] (2,-1.7);
					
					\draw[blue,->] (0,.5) to [bend left] (1,.5);
					\draw[blue, bend at end,<-] (0,.3) to [bend right] (1,.3);
					
					\draw[red,<-] (1.5,.5) to [bend left] (2.5,.5);
					\draw[red, bend at end,->] (1.5,.3) to [bend right] (2.5,.3);
					
					\filldraw[blue] (0,0) circle (1.5pt);
					\filldraw (0.5,-0.5) circle (2.5pt);
					\filldraw[blue] (1,0) circle (1.5pt);
					\filldraw[red] (1.5,0) circle (1.5pt);
					\filldraw (2,-0.5) circle (2.5pt);
					\filldraw[red] (2.5,0) circle (1.5pt);
					\filldraw (1.25,-1) circle (2.5pt);
				\end{scope}
				
			\end{tikzpicture}
			
			\
			
			\begin{tikzpicture}[line width=1pt,yscale=.6]
				
				\draw[white,line width=5pt] (-0.95,0) to[out=-60, in=120] (5.55,-7)   (2.95,0) to[out=-60, in=120] (9.45,-7);
				\draw (-0.95,0) to[out=-60, in=120] (5.55,-7)   (2.95,0) to[out=-60, in=120] (9.45,-7);
				
				\draw (1,0) ellipse (55.5pt and 7.5pt);

				\draw[white,line width=5pt] (5.55,0) to[out=-120, in=60] (-.95,-7)   (9.45,0) to[out=-120, in=60] (2.95,-7);
				\draw (5.55,0) to[out=-120, in=60] (-.95,-7)   (9.45,0) to[out=-120, in=60] (2.95,-7);
				
				\draw (7.5,0) ellipse (55.5pt and 8pt);
				
				\draw (7.5,-7) ellipse (55.49pt and 8pt);
				
				\draw (.999,-7) ellipse (55.5pt and 8pt);
				
				\begin{scope}[yscale=10/6,yshift=-.1\textheight,xshift=.725\linewidth,scale=1.2]
					
					\draw (0.5,-0.5) -- (1.25,-1) -- (2,-.5);

					\draw[->] (0.5,.1) to [bend left] (2,.1);
					\draw[bend at end,<-] (0.5,-.1) to [bend right] (2,-.1);
					
					\filldraw (0.5,-0.5) circle (2.5pt);
					\filldraw (2,-0.5) circle (2.5pt);
					\filldraw (1.25,-1) circle (2.5pt);
				\end{scope}
				
			\end{tikzpicture}
			\caption{A fiber bundle decomposition of a braid in $\UConf_4(S^1,\RR^2)$, represented in a braid-like form. On the right are the associated automorphisms of the planar abstract trees of the Jordan configurations.} 
			\label{fig:fiber_bundle_braid}
		\end{figure}
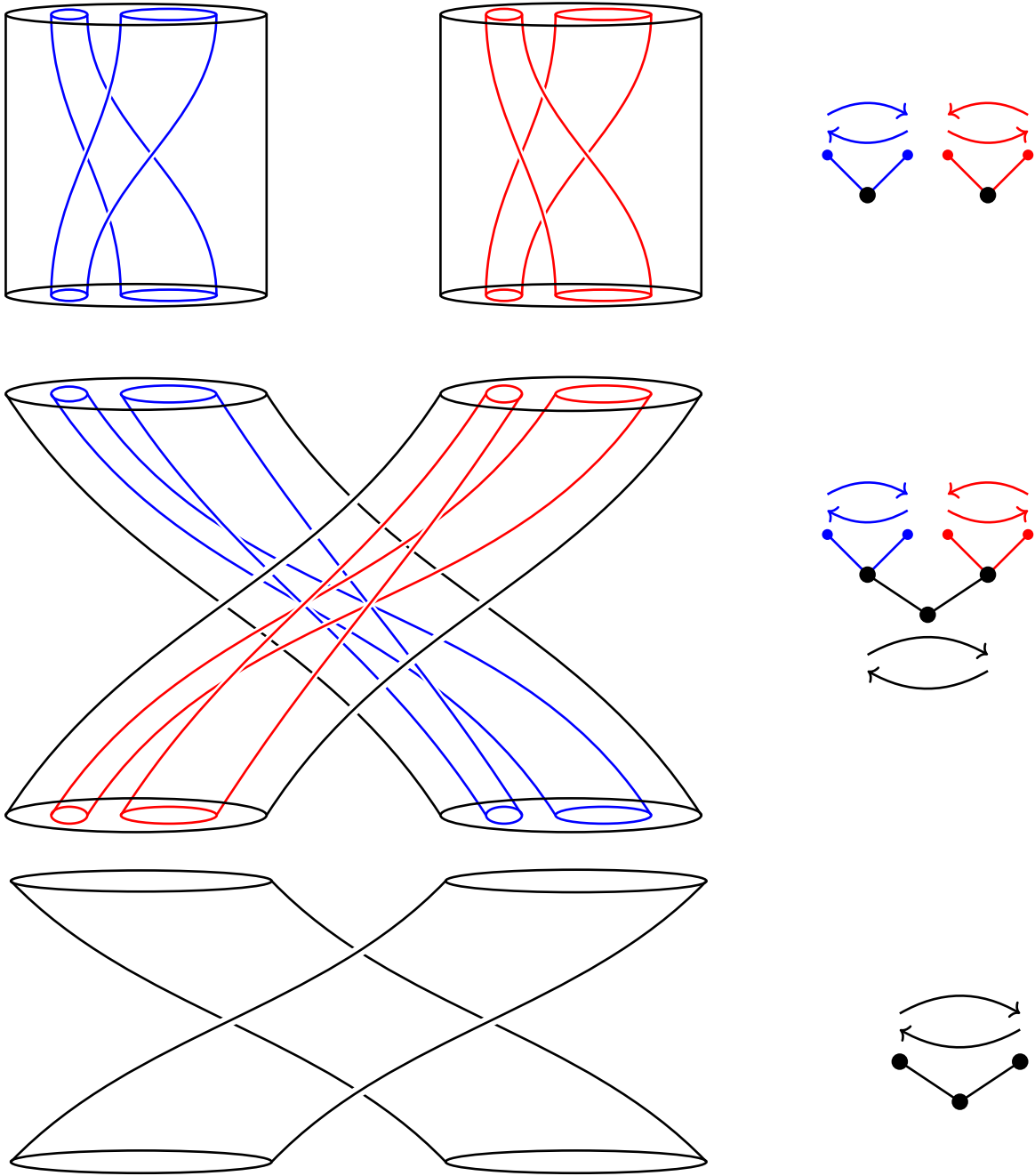
		
	Given a labeled or unlabeled tree $T$, when we write $\PBAut(T)$ and $\BAut(T)$ we assume an arbitrary planar ordering of the children of vertices.
	
	\begin{corollary}\label{cor:class}
		For a tree $T$, the space $\UJ_T(\RR^2)$ is a $K(\BAut(T),1)$. Similarly, for a labeled tree $T$, the space $\J_T(\RR^2)$ is a $K(\PBAut(T),1)$.
	\end{corollary}
	\begin{proof}
		This follows immediately from Theorems~\ref{thrm:confcircunlabclassifying} and~\ref{thrm:confcirclabclassifying} and Definition~\ref{dfn:jt}.
	\end{proof}
	
	Recall from Section~\ref{sec:rounding} the partial ordering of $\{\varnothing,1,\dots,n\}$ given by $T(f)$, and recall $d_j$, the quantity of $i\neq\varnothing$ such that $j\prec i$. Fix a Jordan configuration with abstract tree $T$, denoted $\kappa_T=\{[f_{T,1}],\dots,[f_{T,n}]\}$.
	Similarly to Theorems~\ref{thrm:confcircunlabclassifying}~and~\ref{thrm:confcirclabclassifying}, we define the map $\pi:\UJ_T(X)\to \UJ_m(X)$ to send $\{[f_1],\dots,[f_n]\}$ to $\{[f_i]\mid d_i=0\}$. 
	The induced map on trees sends $T$ to the subtree of $T$ containing only the root and all of the children of the root, call this tree $\Lambda$, so $\pi$ takes Jordan configurations in $\UJ_T(X)$ to Jordan configurations in $\UJ_{\Lambda}(X)$. 
	
	\begin{definition}
		Define $\kappa_{\Lambda}=\{[f_{\Lambda,1}],\dots,[f_{\Lambda,m}]\}$ to be the image of $\kappa_T$ under $\pi$, and we assume that $\kappa_T$ is indexed so that $[f_{T,i}]=[f_{\Lambda,i}]$ for $i\in\{1,\dots,m\}$. 
		The fiber of $f\in \UJ_{\Lambda}(X)$ is the collection of all Jordan configurations in $\UJ_T(X)$ with un-nested Jordan configurations $f$. 
		Let $F$ be the subspace of $\UJ_T(X)$ consisting of Jordan configurations with un-nested curves $\kappa_{\Lambda}$. 
	\end{definition}
		
		\begin{proposition}
			The map $\pi$ makes $\UJ_T(X)$ a fiber bundle over $\UJ_{\Lambda}(X)$, with fiber $F$.
		\end{proposition}
		\begin{proof}
			Let $f=\{[f_1],\dots,[f_m]\}$ be an element of $\UJ_{\Lambda}(X)$. Fix a path from $f$ to $\kappa_{\Lambda}$ and note that this induces a bijection $f\to\kappa_{\Lambda}$. Without loss of generality let $f$ be indexed so that $[f_i]$ is sent to $[f_{\Lambda,i}]$. For $\varepsilon>0$, let $U$ be the open neighborhood of $f$ consisting of the collection of all 
			Jordan configurations in $\UJ_{\Lambda}(X)$ such that for each $i$ there exists a curve with center within $\varepsilon/2$ of the center of $D(f_i)$.
			Choose $\varepsilon$ to be no larger than the minimum pairwise distance between the centers of the domains that each curve of $f$ bounds, so each curve of $f'=\{[f_1'],\dots,[f_m']\}$ in $U$ is uniquely paired with a curve $[f_i]$ of $f$. Without loss of generality let $f'$ be indexed so that the center of $D(f_i')$ is within $\varepsilon/2$ of the center of $D(f_{i})$ and let any Jordan configuration $\hat{f}=\{[\hat{f}_1],\dots,[\hat{f}_n]\}$ in $\pi^{-1}(U)$ be indexed so that for $[\hat{f}_i]$ in $\pi(\hat{f})$ the center of $D(\hat{f}_i)$ is within $\varepsilon/2$ of the center of $D(f_{i})$. 
		
			We want to show that $\pi^{-1}(U)$ is homeomorphic to $U\times F$. 
			Recall the continuous choice of $\gamma_{f_i}\in\Emb(B^2,X)$ from  the construction in subsection~\ref{sec:proof} guaranteed by Lemma~\ref{lem:conformalmapping}.
			Define $\phi:\pi^{-1}(U)\to F$ to send $\hat{f}=\{[\hat{f}_1],\dots,[\hat{f}_n]\}$ to $\{[{f}_{\Lambda,i}],[\gamma_{f_{\Lambda,i}}\circ\gamma^{-1}_{\hat{f}_{i}}\circ\hat{f}_{j}]\mid d_i=0, j\preceq i\}$. 
			We claim $\pi\times\phi:\pi^{-1}(U)\to U\times F$ is the desired homeomorphism. Let $(f',\overline{f})$ be an element of $U\times F$. 
			So we have $f'=\{[f_1'],\dots,[f_m']\}$ with the center of $D(f_i')$ within $\varepsilon/2$ of the center of $D(f_i)$, the configuration $\overline{f}=\kappa_{\Lambda}\cup\{[\overline{f}_{m+1}],\dots,[\overline{f}_{n}]\}$ has abstract tree $T$, 
			and $[\overline{f}_j]$ is nested inside $[f_{\Lambda,i}]$ for $j\in\{m+1,\dots,n\}$ and some $i\in\{1,\dots,m\}$, written as $j\prec i$. 
			The inverse of $\pi\times\phi$ sends $(f',\overline{f})$ to $f'\cup\{[\gamma_{f_i'}\circ\gamma^{-1}_{{f}_{\Lambda,i}}\circ\overline{f}_{j}]\mid d_i=0,j\prec i\}$. 
			Hence, we have a fiber bundle $\pi:\UJ_T(X)\to \UJ_{\Lambda}(X)$ with fiber $F$.
		\end{proof}
	
	\begin{theoremC}
		Let $T$ be a tree with $n$ non-root vertices. With $\{v_1,\dots,v_m\}$ the set of children of the root of $T$, let $\Pi$ be the partition of $\{1,\dots,m\}$ such that $i$ and $j$ share a block of $\Pi$ if and only if $T(v_i)=T(v_j)$. Then $\UJ_T(X)$ is aspherical, and its fundamental group is isomorphic to $$(\BAut(T(v_1))\times\cdots\times\BAut(T(v_m)))\rtimes B_m^{\Pi}(X).$$ Hence, $\UJ_T(X)$ is a $K(\pi,1)$.
	\end{theoremC}
	\begin{proof}
		From the above fiber bundle comes a long exact sequence of homotopy groups:
		$$\cdots\to\pi_k(F,\kappa_T)\to\pi_k(\UJ_T(X),\kappa_T)\to\pi_k(\UJ_{\Lambda}(X),\kappa_{\Lambda})\to\pi_{k-1}(F,\kappa_T)\to\cdots.$$
		Notice that $F$ is homeomorphic to $\UJ_{T(v_1)}(\interior(B^2))\times\dots\times \UJ_{T(v_m)}(\interior(B^2))$ and is aspherical with fundamental group $\BAut(T(v_1))\times\cdots\times \BAut(T(v_m))$ by Theorem~\ref{thrm:confcircunlabclassifying} and Corollary~\ref{cor:he}. Also, $\UJ_{\Lambda}(X)$ is aspherical with fundamental group $B_m(X)$ by Corollary~\ref{cor:hebetweenconfpointandnonnestedcomponents}.
		Hence, $\pi_k(\UJ_T(X),\kappa_T)$ is trivial when $k\geq2$, which yields: 
		$$0\to \prod^m_i\BAut(T(v_i))\to\pi_1(\UJ_T(X),\kappa_T)\to B_m(X)\to\pi_0(F,\kappa_T)\to0.$$
		An element of $\pi_1(\UJ_{\Lambda}(X),\kappa_{\Lambda})$ is in the image of $\pi_1(\UJ_T(X),\kappa_T)$ only if it is represented by a path such that for each $i$ with $d_i=0$, the un-nested curve $[f_i]$ ends up at an un-nested curve $[f_j]$ with $d_j=0$. 
		However, the curves nested inside $[f_i]$ and $[f_j]$ must also permute with each other; this can happen if and only if $T(v_i)$ and $T(v_j)$ are the same tree. 
		Let $\Pi$ be the partition of $\{1,\dots,m\}$ where $i$ and $j$ share a block if and only if $T(v_i)=T(v_j)$.
		Thus, viewing $\pi_1(\UJ_{\Lambda}(X),\kappa_{\Lambda})$ as $B_m(X)$, the image of the map $\pi_1(\UJ_{T}(X),\kappa_T) \to \pi_1(\UJ_{\Lambda}(X),\kappa_{\Lambda})$ is isomorphic to $B_m^{\Pi}(X)$. 
		This yields the short exact sequence
		$$0\to \prod^m_i\BAut(T(v_i))\to\pi_1(\UJ_T(X),\kappa_T)\to B_m^{\Pi}(X)\to0.$$ 
		
		Any element in $B^{\Pi}_m(X)\leq\pi_1(\UJ_{\Lambda}(X),\kappa_{\Lambda})$ is represented by a loop $p_t=\{[f_{1,t}],\dots,[f_{m,t}]\}$ based at $\kappa_{\Lambda}$. Assuming that the indices of $\kappa_{\Lambda}$ and $\kappa_T$ agree with $p_t$, that is $[f_{T,i}]=[f_{\Lambda,i}]=[f_{i,t}]$,
		we get an induced homomorphism, $B_m^{\Pi}(X)\to\pi_1(\UJ_T(X),\kappa_T)$, by sending $p_t$ to $p_t\cup\{[\gamma_{f_{i,t}}\circ\gamma^{-1}_{f_{\Lambda,i}}\circ f_{T,j}]\mid d_i=0,j\preceq i\}$. This homomorphism splits the above sequence and yields the result
		\[
		\pi_1(\UJ_T(X),\kappa_T)\cong(\BAut(T(v_1))\times\cdots\times \BAut(T(v_m))) \rtimes B_m^{\Pi}(X).
		\]
		The action of $B_m^{\Pi}(X)$ on $(\BAut(T(v_1))\times\cdots\times \BAut(T(v_m)))$ induced by the splitting is given by the projection $B_m^{\Pi}(X)\to\Sigma_m^{\Pi}$.
		
		Given a fiber bundle $E\to B$ with $B$ path connected and both the base and fiber are homotopy equivalent to CW-complexes, Theorem~\cite[5.4.2]{fritsch1990cellular} guarantees that the total space $E$ is homotopy equivalent to a CW-complex. 
		Note that $\pi:\UJ_T(X)\to\UJ_{\Lambda}(X)$ has a path connected base and the fiber is homeomorphic to a product of spaces homotopy equivalent to CW-complexes by Theorem~\ref{thrm:confcircunlabclassifying}~and~Corollary~\ref{cor:he}.
		The base space is homotopy equivalent to a CW-complex by Corollary~\ref{cor:hebetweenconfpointandnonnestedcomponents} and the fact that configuration spaces of points in surfaces have the type of CW-complexes; this can be proved using Theorem~\cite[5.4.2]{fritsch1990cellular} and induction on the fundamental sequence of fibrations \cite{fadell1962configuration}.
	\end{proof}
	
	\begin{theorem}\label{thrm:jorconflabclassifying}
		Let $T$ be a labeled tree with $n$ non-root vertices. With $\{v_1,\dots,v_m\}$ the set of children of the root of $T$, let $\Pi$ be the partition of $\{1,\dots,m\}$ such that $i$ and $j$ share a block of $\Pi$ if and only if $T(v_i)\cong T(v_j)$. Then $\J_T(X)$ is aspherical, and its fundamental group is isomorphic to $\PBAut(T(v_1))\times\cdots\times\PBAut(T(v_m))\times PB_m(X)$. Hence, by Lemma~\ref{lem:labeledBAut}, we have $\Conf_T(S^1,\RR^2)$ is a $K(\BAut(T),1)$,  for the labeled tree $T$.
	\end{theorem}
	\begin{proof}
		Since $\J_T(X)$ is a covering space of $\UJ_T(X)$, which is aspherical by Theorem~\ref{thrm:classifying}, we know that $\J_T(X)$ is also aspherical.
		
		The deck group of this cover is the subgroup of $\Sigma_n$ that stabilizes the path connected component $\J_T(X)$ in $\J_n(X)$ and fixes the underlying sets of labeled Jordan configurations of $\J_n(X)$. In particular, the fiber of the fixed unlabeled Jordan configuration of $U(T)$, the underlying unlabeled tree of $T$, is the collection of all labelings of this configuration such that the nesting order of the labels is preserved. 
		It follows that the deck group is $\Aut(U(T))$ and it acts transitively on the fiber, and hence this is a regular covering. By Proposition~1.39~in~\cite{hatcher02}, $\pi_1(\J_T(X),\kappa_T)$ is isomorphic to the kernel of $\pi\colon \BAut(T)\to \Aut(T)$, which is $\PBAut(T)$.
	\end{proof}
	
	\textbf{Further Study.} An immediate generalization would be to consider the entire space of simple closed curves in surfaces, including essential curves and curves that bound marked points and boundary components. 
	We would also like to repeat the results in this paper for the Jordan configurations in a sphere, however there is not a unique disk that any curve bounds so we cannot repeat this paper's arguments; Belegradek and Ghomi show that the space of unparameterized domains in a sphere deformation retracts onto the space of hemispheres \cite{belegradek2025point}.
	Generalizing to Jordan configurations in other two dimensional spaces should also be interesting; for example, given a graph, $\Gamma$, we would like to study the space of Jordan configurations in $\Gamma\times \RR$.

	\bibliographystyle{alpha}
	\bibliography{biblio}
\end{document}